\documentclass[a4paper,reqno,11pt]{amsart}
\usepackage{amsmath}
\usepackage{amsthm}
\usepackage[hmarginratio=1:1]{geometry}
\usepackage{anyfontsize}
\usepackage{kpfonts}
\usepackage[T1]{fontenc}
\usepackage{hyperref}
\hypersetup{
    colorlinks=true,
    linkcolor=blue,
    filecolor=magenta,      
    urlcolor=cyan,
    citecolor=magenta,
}
\usepackage{xcolor}
\usepackage{enumitem}
\usepackage{microtype}
\usepackage{graphicx}
\usepackage{subcaption}

\allowdisplaybreaks
\numberwithin{equation}{section}
\newtheorem{theorem}{Theorem}[section]
\newtheorem{lemma}[theorem]{Lemma}
\newtheorem{proposition}[theorem]{Proposition}

\theoremstyle{definition}
\newtheorem{remark}{Remark}

\title[Stieltjes continued fractions related to two automatic sequences]{Stieltjes continued fractions related to 
the Paperfolding sequence and Rudin-Shapiro sequence}
\author{Wen Wu}
\address{School of Mathematics\\
South China University of Technology\\
Guangzhou 510640, China}
\email{wuwen@scut.edu.cn}
\date{}

\begin{document}
    \begin{abstract}
    We investigate two Stieltjes continued fractions given by the paperfolding sequence and the Rudin-Shapiro sequence. By explicitly describing certain subsequences of the convergents \(P_n(x)/Q_n(x)\) modulo \(4\), we give the formal power series expansions (modulo \(4\)) of these two continued fractions and prove that they are congruent modulo \(4\) to algebraic series in \(\mathbb{Z}[[x]]\). Therefore, the coefficient sequences of the formal power series expansions are \(2\)-automatic. Write \(Q_{n}(x)=\sum_{i\ge 0}a_{n,i}x^{i}\). Then \((Q_{n}(x))_{n\ge 0}\) defines a two-dimensional coefficient sequence \((a_{n,i})_{n,i\ge 0}\). We prove that the coefficient sequences \((a_{n,i}\mod 4)_{n\ge 0}\) introduced by both \((Q_{n}(x))_{n\ge 0}\) and \((P_{n}(x))_{n\ge 0}\) are \(2\)-automatic for all \(i\ge 0\). Moreover, the pictures of these two dimensional coefficient sequences modulo \(4\) present a kind of self-similar phenomenon.
    \end{abstract}
    \maketitle

    \setcounter{tocdepth}{1}
    \tableofcontents

    \section{Introduction}
    In number theory, a fascinating topic is to discuss representations of real numbers. To understand a real number, especially an irrational number, we usually consider its base \(k\) representation, where \(k\ge 2\) is an integer. For example, the work of finding the decimal representation of \(\pi\) is still ongoing. Besides the decimal representation, people are also interested in finding other representations for real numbers, such as the continued fraction expansions. One reason to do this is that for some real numbers, like $\pi$ or $e$, the decimal expansion is irregular but the continued fraction expansion is elegant; see for example \cite{La99, Olds70}. 
    
    This motivates the study of describing the continued fraction expansions of relatively simple power series. A class of simple power series is the power series whose coefficient sequences are automatic sequences. Let \(\mathbf{u}=(u_{n})_{n\ge 0}\) be a sequence with values in \(\mathbb{F}_{q}\). Christol's theorem \cite{Ch79, CKMR80} gives a sufficient and necessary condition of algebraicity for the formal power series \(f(x)=\sum_{n\ge 0}u_nx^{n}\) and automaticity for the sequence \(\mathbf{u}\). Suppose that \(f(x)\) is algebraic over \(\mathbb{F}_{q}(x)\). Then \(\mathbf{u}\) can be recognized as the base \(b\) expansion of the real number \(f(1/b)\) where \(2\leq b< q\) is an integer. If \(u\) is a \(q\)-automatic sequence, then a deep result by Adamczewski and Bugeaud \cite{AB07} showed that \(f(1/b)\) is either transcendental or rational. When \(f(1/b)\) is transcendental, in the study of the Diophantine properties of \(f(1/b)\), it is natural to ask if its continued fraction expansion has bounded partial quotients or not. Van der Poorten and Shallit showed in \cite{PS92} that \(S_{\infty}(1/2)\) consists of partial quotients \(1\) and \(2\), where \(S_{\infty}(x)=\sum_{k=0}^{\infty}x^{2^{k}}\) satisfies the functional equation \[S_{\infty}(x^{2})=S_{\infty}(x)-x.\] Let \(\mathbf{t}=t_0t_1t_2\dots\) be the Thue-Morse sequence on \(\{0,1\}\) given by the recurrence relations $t_0=0$ and for all integer $n\ge 0$, 
\begin{align*}
t_{2n}=t_n\quad \text{and}\quad t_{2n+1}=t_{n}.
\end{align*}
Allouche and Shallit \cite[Open problem 9, p. 403]{AS03} asked whether the Thue-Morse constant \(\tau_{TM}=\sum_{n\ge 0}\frac{t_n}{2^{n+1}}\) has bounded partial quotients. Bugeaud and Queff\'{e}lec \cite{BQ13} showed that the continued fraction expansion of  \(\tau_{TM}\) has infinitely many partial quotients equal to \(4\) or \(5\) and infinitely many partial quotients greater than or equal to \(50\). Badziahin and Zorin \cite{BZ15} proved that the Thue-Morse constant \(\tau_{TM}\) is not badly approximable. Namely, its partial quotients are not bounded. 
    
    For the continued fraction expansions of algebraic elements \(f(x)\), there are not many results. Baum and Sweet \cite{BS76} (and \cite{BS77}) gave an example of an algebraic element of degree \(3\) (and \(k\)) with a bounded (i.e. bounded degree) continued fraction expansion; the coefficient sequence of that algebraic series of degree \(3\) is now known as the Baum-Sweet sequence. In \cite{AFP91}, Allouche, Mend\'{e}s France and van der Poorten showed that functions given by certain infinite products have linear partial quotients. In \cite{van93, van98}, van der Poorten also studied continued fraction expansions for other infinite products. In 2016, Han \cite{Han16} proved an analogue of Lagrange's theorem for Hankel continued fractions of quadratic power series on finite fields.


    Now we consider the continued fraction given an automatic sequence as the sequence of partial quotients. Bugeaud \cite{Bug13} showed that the continued fraction expansion for any algebraic number of degree at least three cannot be generated by a finite automaton. In this paper, we study the formal power expansions of two automatic Stieltjes continued fractions. 
    We assign any finite word \(\mathbf{w}=w_0w_1\cdots w_n\in\{-1,1\}^{n+1}\) to the rational fraction 
    \[\text{Stiel}_{\mathbf{w}}(x):=\cfrac{w_{0}x}{1+\cfrac{w_{1}x}{\cfrac{\ddots}{1+\cfrac{w_{n-1}x}{w_{n}x}}}}.\] 
    An infinite word \(\mathbf{a}=a_{0}a_{1}a_{2}\dots\in\{-1,1\}^{\infty}\)  defines an infinite Stieltjes continued fraction by
    \[\text{Stiel}_{\mathbf{a}}(x):=\lim_{n\to\infty}\text{Stiel}_{\mathbf{a}|_{n}}(x) = 
    \cfrac{a_{0}x}{1+\cfrac{a_{1}x}{1+\cfrac{a_{2}x}{1+\cfrac{a_{3}x}{1+\cfrac{a_{4}x}{\ddots}}}}}\]
    where \(\mathbf{a}|_{n}:=a_0a_1\cdots a_n\) for all \(n\ge 0\). The fraction \(\text{Stiel}_{\mathbf{a}|_{n}}(x)\) is called the \(n\)th \emph{ convergent} of \(\text{Stiel}_{\mathbf{a}}(x)\). For a detail discussion of Stieltjes continued fractions, see \cite{Wall48}. We denote by \(\mathrm{Stiel}_{\mathbf{a}}\) the coefficient sequence of the formal power series expansion of \(\mathrm{Stiel}_{\mathbf{a}}(x)\). Namely, \(\mathrm{Stiel}_{\mathbf{a}}:=(b_{n})_{n\ge 0}\) where \(\mathrm{Stiel}_{\mathbf{a}}(x)=\sum_{n\ge 0}b_{n}x^{n}\in\mathbb{Z}[[x]]\).

    In \cite{HH19}, Han and Hu proved that the Stieltjes continued fractions given by the Thue-Morse sequence and the period doubling sequence are congruent modulo $4$ to algebraic series in \(\mathbb{Z}[[x]]\). Here, we investigate the paperfolding sequence \(\mathbf{p}\) and  Rudin-Shapiro sequence \(\mathbf{r}\). We show that the Stieltjes continued fractions \(\mathrm{Stiel}_{\mathbf{p}}(x)\) and \(\mathrm{Stiel}_{\mathbf{r}}(x)\) are both congruent modulo $4$ to algebraic series in \(\mathbb{Z}[[x]]\). Our results are the following.
    \begin{theorem}\label{thm:01}
    The Stieltjes continued fraction given by the paperfolding sequence is congruent modulo $4$ to an algebraic series in \(\mathbb{Z}[[x]]\). Precisely, \[\mathrm{Stiel}_{\mathbf{p}}(x) \equiv 2x + (3x+2x^{3})\phi(x)\quad (\bmod~4),\] where \(\phi(x)=\sum_{n\ge 0}C_{n}x^{n}\) and $C_n=\frac{1}{n+1}\binom{2n}{n}$ is the $n$th Catalan number.  Moreover, \(\mathrm{Stiel}_{\mathbf{p}}\) modulo \(4\) is \(2\)-automatic.
    \end{theorem}
    
    \begin{theorem}\label{thm:02}
    The Stieltjes continued fraction given by the Rudin-Shapiro sequence is congruent modulo $4$ to an algebraic series in \(\mathbb{Z}[[x]]\). Namely, \[\mathrm{Stiel}_{\mathbf{r}}(x)\equiv x+2x^{2}+ 2x^{3} +(3x+2x^{3})\phi(x)+x\sqrt{1-4x\phi(x)}\quad (\bmod~4).\] Moreover, \(\mathrm{Stiel}_{\mathbf{r}}\) modulo \(4\) is \(2\)-automatic.
    \end{theorem}
    
     Based on the result in \cite{HH19} and our results, it is natural to ask that if one can characterize those automatic sequences \(\mathbf{a}\) on the alphabet \(\{-1,1\}\) such that \(\mathrm{Stiel}_{\mathbf{a}}(x)\) is congruent  modulo \(4\) to an algebraic series in \(\mathbb{Z}[[x]]\).

    The paper is organized as follows. In Section \ref{sec:pre}, we give the definition of \(k\)-automatic sequences and introduce the paperfolding sequence and Rudin-Shapiro sequence. In Section \ref{sec:coeff}, we discuss the coefficient sequences for both the numerator and the denominator of the convergents. Visualizations of these coefficient sequences are also provided. In Section \ref{sec:paper}, we prove Theorem \ref{thm:01}. In Section \ref{sec:rudin}, we prove Theorem \ref{thm:02}.

    \section{Preliminaries}\label{sec:pre}
    \textbf{Substitution and coding.} Let \(\mathcal{A}\) be a finite set, called an \emph{alphabet}. The collection of all words on the alphabet \(\mathcal{A}\) of length \(n\) is denoted by \(\mathcal{A}^{n}\), where \(n\ge 1\) is an integer. In addition, let \(\varepsilon\) be the empty word and \(\mathcal{A}^{0}=\{\varepsilon\}\). The set of all finite words is \(\mathcal{A}^{*}:=\cup_{n\ge 0}\mathcal{A}^{n}\). The \emph{concatenation} of two finite words \(w=w_0w_1\cdots w_n\) and \(v=v_0v_1\cdots v_m\) is the finite word \(wv=w_0w_1\cdots w_nv_0v_1\cdots v_m\). The set \(\mathcal{A}^{*}\) together with concatenation becomes a monoid. A \emph{substitution} is a morphism \(\sigma:\mathcal{A}\to\mathcal{A}^{*}\) which can be extended to \(\mathcal{A}^{*}\). Let \(\mathcal{B}\) be another alphabet. A morphism \(\rho:\mathcal{A}^{*}\to \mathcal{B}^{*}\) is called a \emph{coding} if \(\rho(a)\in\mathcal{B}\) for all \(a\in\mathcal{A}\).
    
    \medskip
    \textbf{Automatic sequences.} Let \(\mathbf{a}=(a_{n})_{n\ge 0}\) be an infinite sequence on an alphabet \(\mathcal{A}\). Let \(k\ge 2\) be an integer. The \(k\)-\emph{kernel} of \(\mathbf{a}\) is the set of subsequences 
    \[K_{k}(\mathbf{a}):=\left\{(a_{k^{i}n+j}) : i\ge 0\ \text{and}\ 0\leq j< k^{i}\right\}.\]
    The sequence \(\mathbf{a}\) is \(k\)-\emph{automatic} if and only if its \(k\)-kernel \(K_{k}(\mathbf{a})\) is finite. 
    
    The following theorem of Christol says the coefficient sequence of algebraic formal power series is automatic.
    \begin{theorem}[Christol et al. \cite{CKMR80}]
        Let \((a_n)_{n\ge 0}\) be a sequence of elements in \(\mathbb{F}_p\). Then \(\sum_{n\ge 0}a_n x^{n}\) is algebraic over \(\mathbb{F}_{p}(x)\)  if and only if \((a_n)_{n\ge 0}\) is \(p\)-automatic.
    \end{theorem}
Denef and Lipshitz \cite{DL87} extend Christol's theorem in the following way; for details see \cite[Theorem 3.1 \& 4.1]{DL87}. 
   \begin{theorem}[Denef and Lipshitz \cite{DL87}]\label{thm:ckmr}
   If the power series $f(x_1,\dots,x_k)\in\mathbb{Z}_p[[x_1,\dots,x_k]]$ is algebraic over $\mathbb{Z}_p[x_1,\dots,x_k]$, then for all integer $s\ge 1$, the coefficient sequence of $f$ ($\mathrm{mod}~p^{s}$) is $p$-automatic.
   \end{theorem}

    For our purpose, we collect some known properties of \(k\)-automatic sequences in the following lemma; for details, see Theorem 5.4.1, Theorem 5.4.3, Corollary 5.4.5 and Theorem 6.8.1 in \cite{AS03}. 
    \begin{lemma}[see \cite{AS03}] \label{lem:auto-prop}
        Let \((a_{n})_{n\ge 0}\) and  \((b_{n})_{n\ge 0}\) be two \(k\)-automatic sequences with values in finite sets \(\mathcal{A}\) and \(\mathcal{A}'\) respectively.
        \begin{enumerate}
            \item[(1)] If the sequence \((d_n)_{n\ge 0}\) differs only in finitely many terms from \((a_{n})_{n\ge 0}\), then \((d_n)_{n\ge 0}\) is \(k\)-automatic.
            \item[(2)] Let \(\rho\) be a coding. Then \((\rho(a_n))_{n\ge 0}\) is \(k\)-automatic.
            \item[(3)] Let \(f:\mathcal{A}\times\mathcal{A}'\to\mathcal{A}''\) be any function into the finite set \(\mathcal{A}''\). Then the sequence \((f(a_n,b_n))_{n\ge 0}\) is \(k\)-automatic.
            \item[(4)] For all integers \(s,t\ge 0\), the subsequence \((a_{sn+t})_{n\ge 0}\) is \(k\)-automatic.
        \end{enumerate}
        
    \end{lemma}

    The next result can be used to deal with running sums and running products of \(k\)-automatic sequences.
    \begin{lemma}[Theorem 2 in \cite{AF86}]\label{lem:running}
        Let \(\mathcal{A}\) be an alphabet on which an associative operation \(*\) is defined. Let \((x_{n})_{n\ge 0}\) be a \(k\)-automatic sequence on \(\mathcal{A}\). Then the sequence \((y_{n}:=x_{n-1}*x_{n-2}*\dots * x_0)_{n\ge 1}\) is \(k\)-automatic.
    \end{lemma}

    The paperfolding sequence and the Rudin-Shapiro sequence are two well known \(2\)-automatic sequences. 
    
    \medskip
    \textbf{Paperfolding sequence.} The paperfolding sequence \(\mathbf{p}=(p_n)_{n\ge 0}\) on the alphabet \(\{-1,1\}\) is defined as follows: \(p_0=1\) and for all \(n\ge 1\), 
    \[\begin{cases}
        p_{4n}=1,\\
        p_{4n+2}=-1,\\
        p_{2n+1}=p_{n}.
    \end{cases}\]
    The sequence can also be generated by using the substitution \[\sigma: a\to ab,\ b\to cb,\ c\to ad,\ d\to cd\] and the projection \(\rho: a\to 1,\, b\to 1,\ c\to -1,\ d\to -1\). That is \(\mathbf{p}=\lim_{n\to\infty}\rho(\sigma^{n}(a))\).

    \medskip
    \textbf{Rudin-Shapiro sequence.} The Rudin-Shapiro sequence \(\mathbf{r}=(r_{n})_{n\ge 0}\) on the alphabet \(\left\{-1,1\right\}\) is defined as follows: \(r_0=1\) and for all \(n\ge 1\), 
    \[\begin{cases}
        r_{2n}=r_{n},\\
        r_{2n+1}=(-1)^{n}r_{n}.
    \end{cases}\]
    The sequence can also be generated by using the substitution \[\sigma_{rs}: a\to ab,\ b\to ac,\ c\to db,\ d\to dc\] and the projection \(\rho_{rs}: a\to 1,\, b\to 1,\ c\to -1,\ d\to -1\). That is \(\mathbf{r}=\lim_{n\to\infty}\rho_{rs}(\sigma_{rs}^{n}(a))\).

    \medskip
    \textbf{Hankel determinant.}    
    Let \(\mathbf{b}=(b_{n})_{n\ge 0}\) be an integer sequence. Then, for all \(n\ge 1\), the \(n\)th-order Hankel determinant of \(\mathbf{b}\) is 
    \[H_{n}(\mathbf{b}) : = \det(b_{i+j-2})_{1\leq i,j\leq n}=\begin{vmatrix}
        b_{0} & b_{1} & \dots & b_{n-1}\\
        b_{1} & b_{2} & \dots & b_{n}\\
        \vdots & \vdots & \ddots & \vdots\\
        b_{n-1} & b_{n} & \dots & b_{2n-2}
    \end{vmatrix}.\]
    Heilermann \cite{Hei46} gave the nice connection between the Stieltjes continued fraction and its Hankel determinant: for all \(n\ge 1\), 
    \begin{equation}
        H_{n}(\mathrm{Stiel}_{\mathbf{a}}) = a_{0}^{n}(a_{1}a_{2})^{n-1}(a_{3}a_{4})^{n-2}\dots (a_{2n-3}a_{2n-2}).\label{eq:heil}
    \end{equation}
    The Hankel determinants are expressed in the twice running product of \(\mathbf{a}\). Namely, letting \(b_{n}=\prod_{i=0}^{2n}a_{i}\) for all \(n\ge 0\), then \(H_{n}(\mathrm{Stiel}_{\mathbf{a}})=\prod_{i=0}^{n-1}b_{i}\). 

    When \(\mathbf{a}\in\{-1,1\}^{\infty}\) is \(p\)-automatic, according to Lemma \ref{lem:running}, the running product sequence \((\prod_{i=0}^{n}a_{i})_{n\ge 0}\) is \(p\)-automatic. By Lemma \ref{lem:auto-prop} (4), its subsequence \((\prod_{i=0}^{2n}a_i)_{n\ge 0}\) is also \(p\)-automatic. Using Lemma \ref{lem:running} again, one can see that \((H_{n}(\mathrm{Stiel}_{\mathbf{a}}))_{n\ge 1}\) is a \(p\)-automatic sequence on the alphabet \(\{-1,1\}\). Further, if \(\mathbf{a}\) is a \(p\)-automatic sequence taking values in \(\mathbb{Z}\backslash\{0\}\),  then for any integer \(m\ge 2\),
    \begin{itemize}
        \item the sequence \((H_{n}(\mathrm{Stiel}_{\mathbf{a}})\mod m)_{n\ge 1}\) is \(p\)-automatic.
    \end{itemize} 

    \emph{Notations}. We define for \(n\ge 0\),
    \[S_{n}(x)=\sum_{i=0}^{n}x^{2^{i}},\quad S_{n}^{e}(x)=\sum_{i=0}^{n}x^{2^{2i}},\quad S_{n}^{o}(x)=\sum_{i=0}^{n}x^{2^{2i+1}}\]
    and for \(n\ge 2\), \[\quad T_{n}(x)=\sum_{i=3}^{n}\sum_{k=2}^{i-1}x^{2^{i}+2^{k}}\]
    where \(T_2(x)=0\).

    Throughout the paper, we denote by `\(\equiv_{m}\)' the congruence modulo \(m\), where \(m\ge 2\) is an integer.
    
    We record several useful equalities in the following:
    \begin{align}
        2S_{n}^{2}(x) & \equiv_{4} 2S_{n}(x^{2})\nonumber\\  
        & \equiv_{4} 2(S_{n+1}(x)-x),\label{eq:2s2}\\ 
        2S_{n-1}(x)S_{n}(x) & \equiv_{4} 2(S_{n}^2(x)-x^{2^{n}}S_{n}(x)) \nonumber\\
        & \equiv_{4} 2(S_{n+1}(x)-x)+2x^{2^{n}}S_{n}(x),\label{eq:ss}\\
        2T_{n}(x) & \equiv_{4} 2\sum_{i=3}^{n}\sum_{k=2}^{i-1}x^{2^{i}+2^{k}} \nonumber\\
        & \equiv_{4} 2\sum_{i=3}^{n-1}\sum_{k=2}^{i-1}x^{2^{i}+2^{k}} + 2 \sum_{k=2}^{n-1}x^{2^{n}+2^{k}}\nonumber\\
        & \equiv_{4} 2T_{n-1}(x)+2x^{2^{n}}(S_{n-1}(x)-x-x^{2}),\label{eq:t}\\
        S_{n}^{2}(x) & \equiv_{4} \sum_{i=0}^{n}\sum_{k=0}^{n}x^{2^{i}+2^{k}} \equiv_{4} 2\sum_{i=1}^{n}\sum_{k=0}^{i-1}x^{2^{i}+2^{k}}+\sum_{i=0}^{n}x^{2^{i+1}}\nonumber\\
        & \equiv_{4} (3x+2x^{2}+2x^{3}+2x^{4}) + 2(x+x^{2})S_{n}(x) + S_{n+1}(x) + 2T_{n}(x).\label{eq:s2}
    \end{align}
    Write \(S_{\infty}(x):=\lim\limits_{n\to\infty}S_{n}(x)\), and similarly we define \(S_{\infty}^{e}(x)\), \(S_{\infty}^{o}(x)\) and \(T_{\infty}(x)\).

    \section{Coefficients of convergents}\label{sec:coeff}
    Let \(\mathbf{c}=(c_n)_{n\ge 0}\) be an infinite sequence on the alphabet \(\{-1,1\}\). For all \(n\ge 1\), the \(n\)th convergent of \(\text{Stiel}_{\mathbf{c}}(x)\) is written by \[\frac{P^{\mathbf{c}}_n(x)}{Q^{\mathbf{c}}_{n}(x)}:=\text{Stiel}_{\mathbf{c}|_{n}}(x)\] where \(P^{\mathbf{c}}_{n}(x)\) and \(Q^{\mathbf{c}}_{n}(x)\) are co-prime polynomials in \(\mathbb{Z}[x]\). In addition, we define \(P^{\mathbf{c}}_{0}(x)=c_{0}x\) and \(Q^{\mathbf{c}}_{0}(x)=1\). For simplicity, we shall use \(P_{n}(x)\) and \(Q_{n}(x)\) instead of \(P^{\mathbf{c}}_{n}(x)\) and \(Q^{\mathbf{c}}_{n}(x)\). This will not cause any misunderstanding, since we focus on one sequence at a time in different sections.

    A basic relation between consecutive convergents is that for all \(n\ge 1\),
    \begin{equation}\label{eq:convergents}
        \begin{pmatrix}
            P_{n-1}(x) & P_{n}(x)\\ Q_{n-1}(x) & Q_{n}(x)
        \end{pmatrix} = 
        \begin{pmatrix}
            0 & c_{0}x\\ 1 & 1
        \end{pmatrix}
        \begin{pmatrix}
            0 & c_{1}x\\ 1 & 1
        \end{pmatrix}\dots
        \begin{pmatrix}
            0 & c_{n}x\\ 1 & 1
        \end{pmatrix};
    \end{equation}
    see for example \cite{EW11, Wall48}. The sequences of polynomials \((Q_{n}(x))_{n\ge 1}\) and \((P_{n}(x))_{n\ge 1}\) share the same recurrence relation for \(n\ge 2\),
    \begin{equation}
        F_n(x) = F_{n-1}(x) + c_{n}xF_{n-2}(x), \label{eq:f-rec}
    \end{equation} 
    where \(F\) stands for \(P\) and \(Q\).

    Given the initial condition \(F_0(x),\, F_{1}(x)\in\mathbb{Z}[x]\), we investigate the sequence of polynomials \((F_{n}(x))_{n\ge 0}\) satisfying the recurrence relation \eqref{eq:f-rec}. 
    For all \(n\ge 0\), write 
    \begin{equation}
        F_{n}(x) = \sum_{i\ge 0}a_{n,i}x^{i} \label{eq:ani}
    \end{equation}
    where \(a_{n,i}\in\mathbb{Z}\) for all \(i\ge 0\). In this way, the sequence \((F_{n}(x))_{n\ge 0}\) defines the two-dimensional sequence \((a_{n,i})_{n,i\ge 0}\) taking values in \(\mathbb{Z}\). Note that \(a_{n,i}=0\) for all \(i>\deg(F_{n})\). For all \(n\ge 0\), the sequence \((a_{n,i})_{i\ge 0}\) is eventually constant.

    \begin{proposition}\label{prop:ani}
        Let  \((F_{n}(x))_{n\ge 0}\) be a sequence of polynomials satisfying the recurrence relation \eqref{eq:f-rec} and \((a_{n,i})_{n,i\ge 0}\) is defined by \eqref{eq:ani}. If the sequence \(\mathbf{c}\) is \(p\)-automatic, then for all \(i\ge 0\), the sequence \((a_{n,i})_{n\ge 0}\) modulo \(m\) \((m\ge 2)\) is \(p\)-automatic.
    \end{proposition}

    \begin{proof}
        From \eqref{eq:f-rec}, for all \(n\ge 2\), one has \(a_{n,0} = a_{n-1,0}\) and for all \(i\ge 1\), 
        \begin{equation}
            a_{n,i} = a_{n-1,i} + c_{n}a_{n-2,i-1}. \label{eq:a-rec}
        \end{equation}
        Therefore, \(a_{n,0}=a_{1,0}\) for all \(n\ge 0\). This implies that \((a_{n,0})_{n\ge 0}\) is \(p\)-automatic. 
        
        Using \eqref{eq:a-rec} \(n-2\) times, we have  
        \begin{align*}
            a_{n,1} & = a_{n-1,1} + c_{n}a_{n-2,0} = a_{n-1,1} + c_{n}a_{1,0},\\
            a_{n-1,1} & = a_{n-2,1}+c_{n-1}a_{n-3,0} = a_{n-2,1} + c_{n-1}a_{1,0},\\
            & \dots\\
            a_{3,1} & = a_{2,1} + c_{3}a_{1,0}.
        \end{align*}
        Adding them up, we obtain that \(a_{n,1} = a_{2,1} + a_{1,0}\sum_{j=3}^{n}c_{j}\). According to Lemma \ref{lem:running}, we obtain that \((\sum_{j=3}^{n}c_{j}\ \mathrm{mod}\ m)_{n\ge 3}\)  is \(p\)-automatic. Therefore,  \((a_{n,1}\ \mathrm{mod}\ m)_{n\ge 3}\) is \(p\)-automatic and so is  \((a_{n,1}\ \mathrm{mod}\  m)_{n\ge 0}\).
    
        Now suppose \((a_{n,i}\ \mathrm{mod}\ m)_{n\ge 0}\) (\(i\ge 1\)) is \(p\)-automatic. We show that \((a_{n,i+1}\ \mathrm{mod}\ m)_{n\ge 0}\) is also \(p\)-automatic. Using \eqref{eq:a-rec} as previously, we have 
        \[a_{n,i+1} = a_{2,i+1}+\sum_{j=3}^{n}c_{j}a_{j-2,i}.\]
        Note that \(\mathbf{c}\) is \(p\)-automatic and by the inductive hypothesis, we see that \((a_{n,i}\ \mathrm{mod}\ m)_{n\ge 0}\) is \(p\)-automatic too. Then by Lemma \ref{lem:auto-prop} (3), their product \((c_{j}a_{j-2,i}\ \mathrm{mod}\ m)_{j\geq 3}\) is also \(p\)-automatic. Then by Lemma \ref{lem:running}, the running sum sequence \((\sum_{j=3}^{n}c_{j}a_{j-2,i}\ \mathrm{mod}\ m)_{n\geq 3}\) is \(p\)-automatic. According to Lemma \ref{lem:auto-prop} (1) and (2), the sequence \((a_{n,i+1}\ \mathrm{mod}\ m)_{n\ge 0}\) is \(p\)-automatic.
    \end{proof}
    
    In the following, we focus on the two-dimensional coefficient sequences of \((Q_{n}(x))_{n\ge 1}\) and \((P_{n}(x))_{n\ge 1}\).
    \subsection{Visualization of \((a_{n,i})_{n,i\ge 0}\)}
     Suppose that the two-dimensional sequence \((a_{n,i})_{n,i\ge 0}\) is given by \((Q_{n}(x))_{n\ge 1}\) (or \((P_{n}(x))_{n\ge 1}\)) as in \eqref{eq:ani}.  Then we have  
    \begin{itemize}
        \item[-] for all \(n\ge 0\), the sequence \((a_{n,i})_{i\ge 0}\) is eventually zero.
    \end{itemize}
    Moreover, by Proposition \ref{prop:ani}, we see that for any \(m\ge 2\),
    \begin{itemize}
        \item[-] for all \(i\ge 0\), the sequence \((a_{n,i}\mod m)_{n\ge 0}\) is \(p\)-automatic.
    \end{itemize}
    
    The two-dimensional sequence \((a_{n,i}\mod m)_{n,i\ge 0}\) also presents a kind of self-similar property. We visualize the two-dimensional sequence \((a_{n,i})_{n\ge 1,i\ge 0}\) in the following way: 
    \begin{itemize}
        \item if \(a_{n,i}=1\) (resp. \(2\), \(3\)), then we plot a red (resp. green, blue) square at the position \((n,i)\);
        \item if \(a_{n,i}=1\), then we plot a white square at the position \((n,i)\).
    \end{itemize}
    Namely, a square in white (resp. red, green, blue) at the position \((n,i)\) indicates that the coefficient modulo \(4\) of the term \(x^{i}\) in \(Q_n(x)\) is  \(0\) (resp. \(1,\,2,\,3\)). Figure \ref{fig:paperfolding} and Figure \ref{fig:rudin-shapiro} illustrate the two dimensional coefficient sequences (modulo \(4\)) of \((Q_{n}(x))_{n\ge 1}\) and \((P_{n}(x))_{n\ge 1}\) for the paperfolding sequence and the Rudin-Shapiro sequence respectively.
    \begin{figure}[htbp]
        \centering
        \begin{subfigure}[t]{.49\textwidth}
            \centering
            \includegraphics[width=\textwidth]{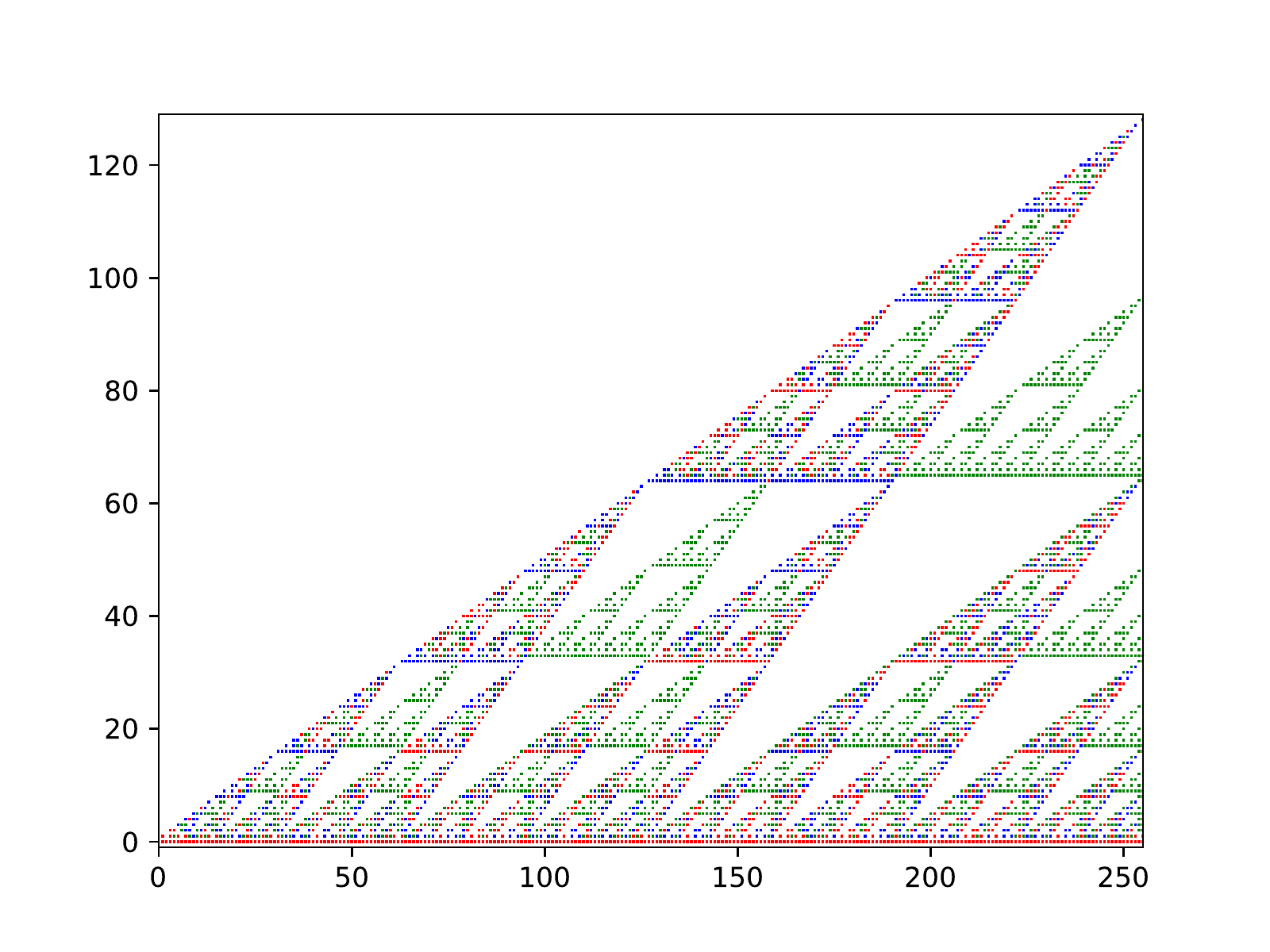}
            \caption{Visualization of \((Q^{\mathbf{p}}_{n}(x))_{n=1}^{256}\) modulo \(4\).}
        \end{subfigure}
        \begin{subfigure}[t]{.49\textwidth}
            \centering
            \includegraphics[width=\textwidth]{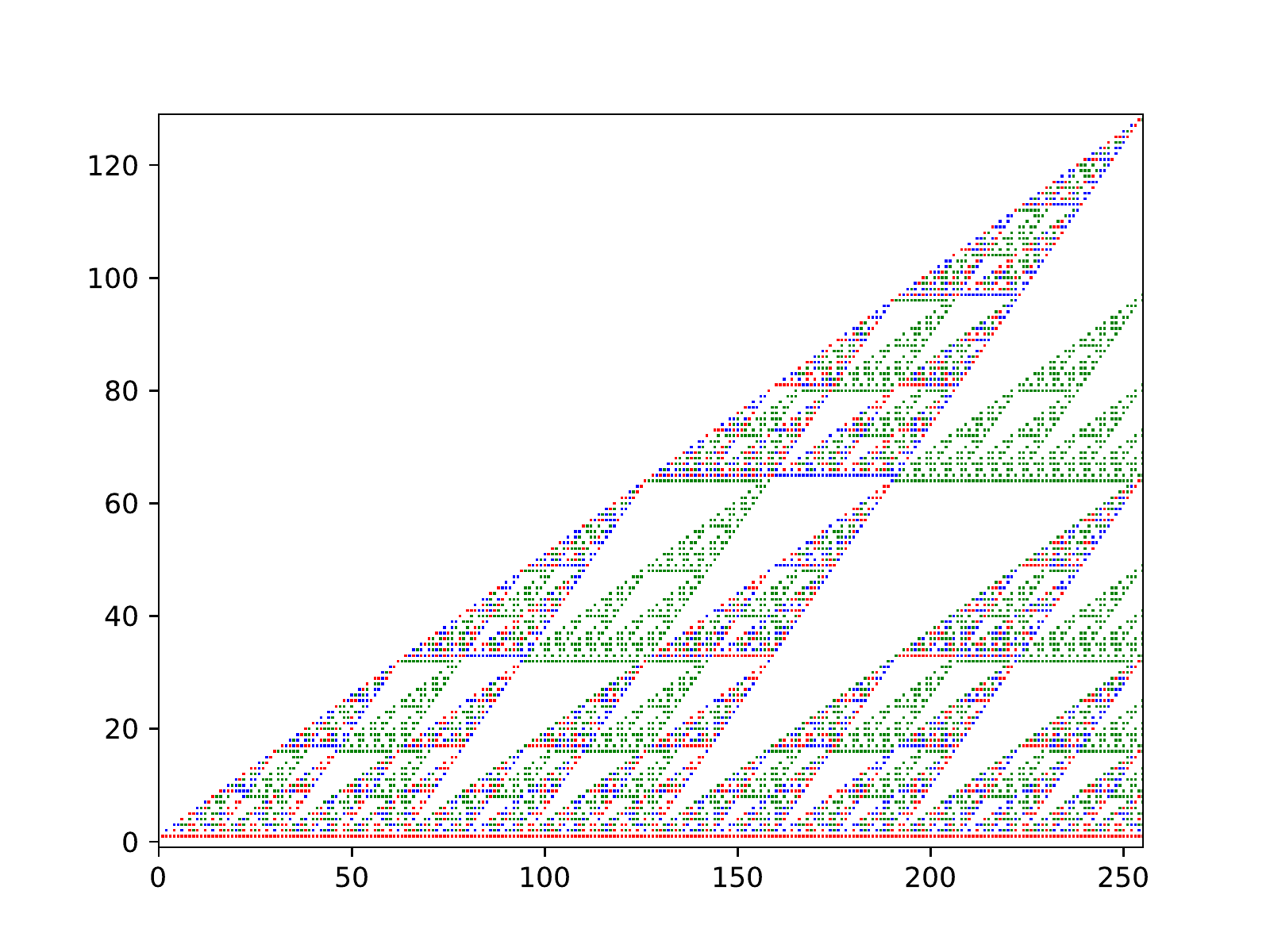}
            \caption{Visualization of \((P^{\mathbf{p}}_{n}(x))_{n=1}^{256}\) modulo \(4\).}
        \end{subfigure}
        \caption{Paperfolding sequence}
        \label{fig:paperfolding}
    \end{figure}

    \begin{figure}[htbp]
        \centering
        \begin{subfigure}[t]{.49\textwidth}
            \centering
            \includegraphics[width=\textwidth]{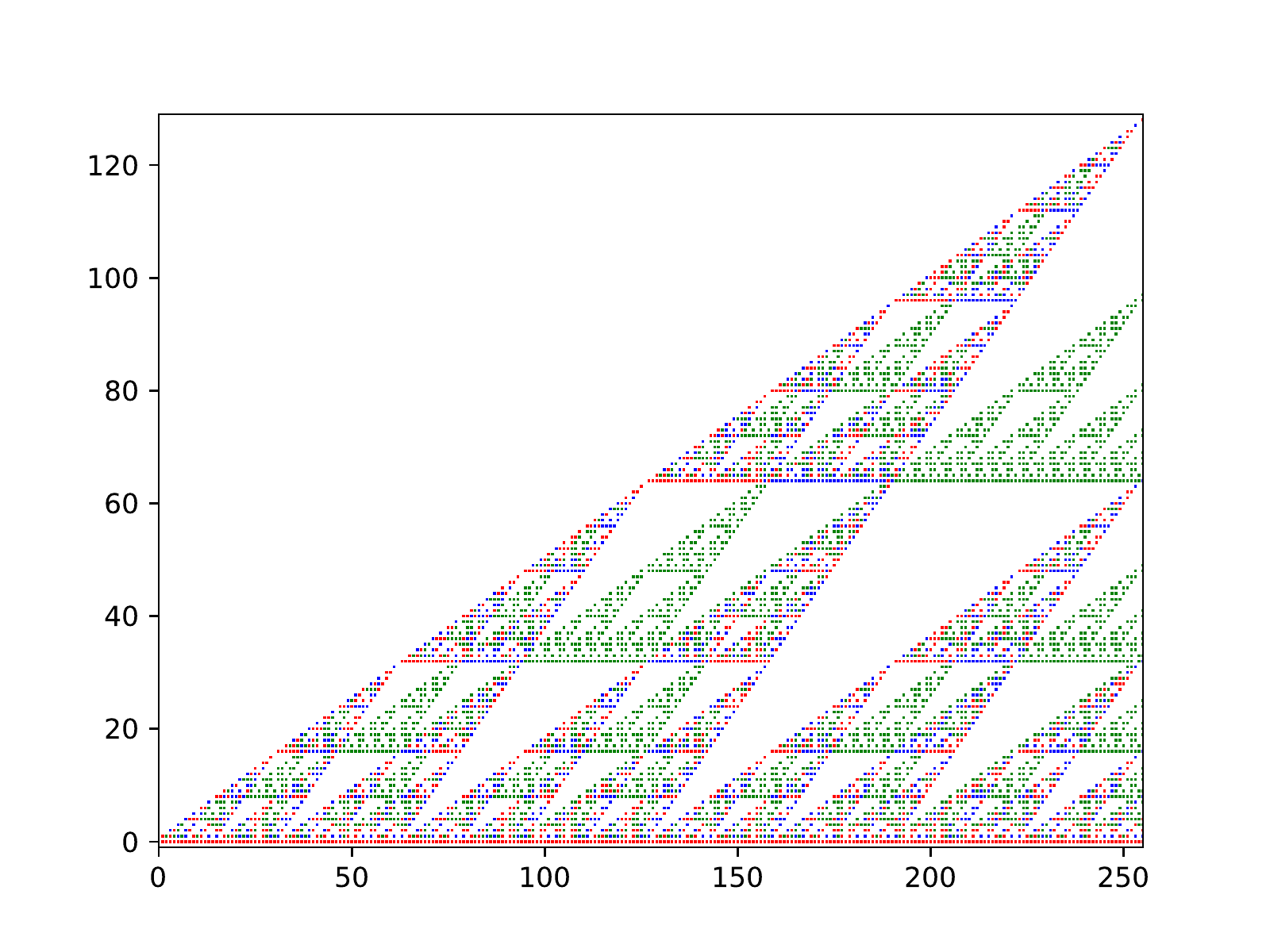}
            \caption{Visualization of \((Q^{\mathbf{r}}_{n}(x))_{n=1}^{256}\) modulo \(4\).}
        \end{subfigure}
        \begin{subfigure}[t]{.49\textwidth}
            \centering
            \includegraphics[width=\textwidth]{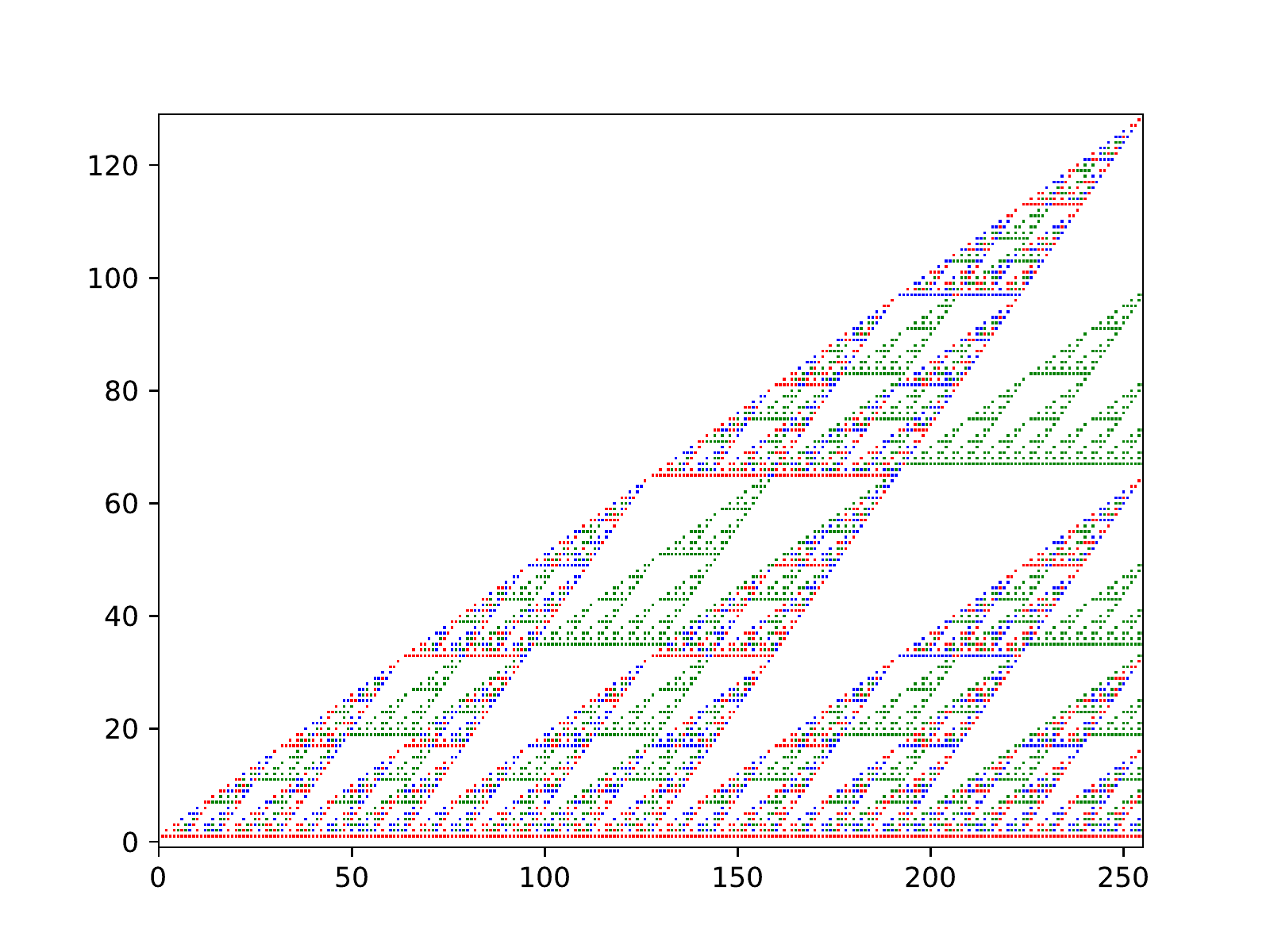}
            \caption{Visualization of \((P^{\mathbf{r}}_{n}(x))_{n=1}^{256}\) modulo \(4\).}
        \end{subfigure}
        \caption{Rudin-Shapiro sequence}
        \label{fig:rudin-shapiro}
    \end{figure}

    \begin{remark}
        The white triangles in Figure \ref{fig:paperfolding} and \ref{fig:rudin-shapiro} represent the block of zeros of \((a_{n,i})_{n\ge 1}\) while modulo \(4\). The occurrence of the triangular shape can be explained by \eqref{eq:a-rec}. Suppose \(a_{n,i}\equiv_4 c\in\mathbb{Z}/4\mathbb{Z}\) for \(n_1\leq n\leq n_1+n_0\). That is we have a horizontal line segment of length \(n_0+1\) in only one color in the picture. Then \(p_{n}a_{n-2,i-1}=a_{n,i}-a_{n-1,i}\equiv_4 0\). Hence \(a_{n,i-1}\equiv_4 0\) for \(n_1-1\leq n\leq n_1+n_0-2\). Namely, below the previous line segment, we have a white horizontal line segment of length \(n_0\). Note that the length of the new line segment shrinks by \(1\). Repeating this process \(n_0\) times, we finally obtain a line segment of length \(1\), i.e., a point in white. Those \(n_0\) line segments form a triangle in white.
    \end{remark}

    \section{Paperfolding sequence and its Stieltjes continued fraction}\label{sec:paper}
    Let \(\mathbf{p}\) be the paperfolding sequence on the alphabet \(\{-1,1\}\). Recall that \[\mathbf{p}=\lim\limits_{n\to\infty}\rho(\sigma^{n}(a))\] where \[\sigma: a\to ab,\ b\to cb,\ c\to ad,\ d\to cd\quad\text{and}\quad \rho: a\to 1,\, b\to 1,\ c\to -1,\ d\to -1.\]
    For all \(n\ge 1\), let \(P_n(x)/Q_{n}(x):=\text{Stiel}_{\mathbf{p}|_{n}}(x)\) be the \(n\)th convergent of \(\text{Stiel}_{\mathbf{p}}(x)\), where \(P_{n}(x)\) and \(Q_{n}(x)\) are co-prime polynomials in \(\mathbb{Z}[x]\). In addition, we define \(P_{0}(x)=p_{0}x\) and \(Q_{0}(x)=1\).  
    For our purpose, we introduce the other two sequences of polynomials
    \(P_{j}^{b}(x)\) and \(Q_{j}^{b}(x)\) for \(j\in\{2^{n}-1\mid n\ge 2\}\cup\{2^{n}-2\mid n\ge 2\}\).  For \(n\ge 2\), set
    \begin{equation*}
        \begin{pmatrix}
            P_{2^{n}-2}^{b}(x) & P_{2^{n}-1}^{b}(x)\\ 
            Q_{2^{n}-2}^{b}(x) & Q_{2^{n}-1}^{b}(x)
        \end{pmatrix} := 
        \begin{pmatrix}
            0 & p_{2^{n}}x\\ 1 & 1
        \end{pmatrix}
        \begin{pmatrix}
            0 & p_{2^{n}+1}x\\ 1 & 1
        \end{pmatrix}\cdots
        \begin{pmatrix}
            0 & p_{2^{n+1}-1}x\\ 1 & 1
        \end{pmatrix}.
    \end{equation*}
    In fact, \(P_{2^{n}-1}^{b}(x)/Q_{2^{n}-1}^{b}(x)=\text{Stiel}_{\rho(\sigma^{n}(b))}(x)\).

    The next result allows us to describe a subsequence of convergents.
    \begin{lemma}
        For all \(n\ge 2\) and \(\ell = 1 \) or \(2\), we have 
        \begin{align}
            P_{2^{n}-\ell}(x) & = P_{2^{n-1}-2}(x)P^{b}_{2^{n-1}-\ell}(x) + P_{2^{n-1}-1}(x)Q^{b}_{2^{n-1}-\ell}(x),\label{eq:p-1}\\
            Q_{2^{n}-\ell}(x) & = Q_{2^{n-1}-2}(x)P^{b}_{2^{n-1}-\ell}(x) + Q_{2^{n-1}-1}(x)Q^{b}_{2^{n-1}-\ell}(x),\label{eq:q-1}\\
            P^{b}_{2^{n}-\ell}(x) & = P_{2^{n}-\ell}(x) + Q^{b}_{2^{n-1}-\ell}(x)\left(2P_{2^{n-1}-2}(x) - 2P_{2^{n-1}-1}(x)\right),\label{eq:pb-1}\\
            Q^{b}_{2^{n}-\ell}(x) & = Q_{2^{n}-\ell}(x) + Q^{b}_{2^{n-1}-\ell}(x)\left(2Q_{2^{n-1}-2}(x) - 2Q_{2^{n-1}-1}(x)\right).\label{eq:qb-1}
        \end{align}
    \end{lemma}
    \begin{proof}
        For \(n\ge 2\), using \eqref{eq:convergents},
        \begin{align*}
            \begin{pmatrix}
                P_{2^{n}-2}(x) & P_{2^{n}-1}(x)\\ Q_{2^{n}-2}(x) & Q_{2^{n}-1}(x)
            \end{pmatrix}  & = 
            \begin{pmatrix}
                0 & p_{0}x\\ 1 & 1
            \end{pmatrix}
            \begin{pmatrix}
                0 & p_{1}x\\ 1 & 1
            \end{pmatrix}\dots
            \begin{pmatrix}
                0 & p_{2^{n}-1}x\\ 1 & 1
            \end{pmatrix}\\
            & =  \begin{pmatrix}
                0 & p_{0}x\\ 1 & 1
            \end{pmatrix}\cdots
            \begin{pmatrix}
                0 & p_{2^{n-1}-1}x\\ 1 & 1
            \end{pmatrix}
            \begin{pmatrix}
                0 & p_{2^{n-1}}x\\ 1 & 1
            \end{pmatrix}\cdots
            \begin{pmatrix}
                0 & p_{2^{n}-1}x\\ 1 & 1
            \end{pmatrix}\\
            & = \begin{pmatrix}
                P_{2^{n-1}-2}(x) & P_{2^{n-1}-1}(x)\\ Q_{2^{n-1}-2}(x) & Q_{2^{n-1}-1}(x)
            \end{pmatrix}\begin{pmatrix}
                P_{2^{n-1}-2}^{b}(x) & P_{2^{n-1}-1}^{b}(x)\\ 
                Q_{2^{n-1}-2}^{b}(x) & Q_{2^{n-1}-1}^{b}(x)
            \end{pmatrix}
        \end{align*}
        which proves \eqref{eq:p-1} and \eqref{eq:q-1}.

        For the remaining two equalities, we need the following decomposition structure of \(\sigma^{n}(a)\) and \(\sigma^{n}(c)\):
        \begin{align*}
            \sigma^{n}(a) & = \sigma^{n-1}(a)\sigma^{n-1}(b)\\
            & =\sigma^{n-1}(a)\sigma^{n-2}(c)\sigma^{n-2}(b) \\
            & \dots \\
            & = \sigma^{n-1}(a)\sigma^{n-2}(c)\sigma^{n-3}(c)\cdots\sigma(c)cb
        \end{align*}
        and 
        \begin{align*}   
            \sigma^{n}(c) & = \sigma^{n-1}(a)\sigma^{n-1}(d) \\
            & =\sigma^{n-1}(a)\sigma^{n-2}(c)\sigma^{n-2}(d)\\
            & \dots \\
            & = \sigma^{n-1}(a)\sigma^{n-2}(c)\sigma^{n-3}(c)\cdots\sigma(c)cd.
        \end{align*}
        So \(\sigma^{n}(a)\) and \(\sigma^{n}(c)\) differ only on the last digit. Moreover, the last digit of \(\rho(\sigma^{n}(a))\) is \(1\) while the last digit of \(\rho(\sigma^{n}(c))\) is \(-1\). Therefore,
        \begin{align*}
            p_{2^{n}}\dots p_{2^{n+1}-1} & = \rho(\sigma^{n}(b)) \\
            & =\rho(\sigma^{n-1}(c)\sigma^{n-1}(b))\\
            & = p_{0}\dots p_{2^{n-1}-2}\rho(d)p_{2^{n-1}}\dots p_{2^{n}-1}.
        \end{align*}
        This fact yields that 
        \begin{align*}
            \begin{pmatrix}
                P_{2^{n}-2}^{b}(x) & P_{2^{n}-1}^{b}(x)\\ 
                Q_{2^{n}-2}^{b}(x) & Q_{2^{n}-1}^{b}(x)
            \end{pmatrix} & = 
            \begin{pmatrix}
                0 & p_{2^{n}}x\\ 1 & 1
            \end{pmatrix}
            \begin{pmatrix}
                0 & p_{2^{n}+1}x\\ 1 & 1
            \end{pmatrix}\dots
            \begin{pmatrix}
                0 & p_{2^{n+1}-1}x\\ 1 & 1
            \end{pmatrix}\\
            & = 
            \begin{pmatrix}
                0 & p_{0}x\\ 1 & 1
            \end{pmatrix}\cdots
            \begin{pmatrix}
                0 & p_{2^{n-1}-2}x\\ 1 & 1
            \end{pmatrix}\cdot
            \begin{pmatrix}
                0 & -x\\ 1 & 1
            \end{pmatrix}\\
            & \qquad\cdot
            \begin{pmatrix}
                0 & p_{2^{n-1}}x\\ 1 & 1
            \end{pmatrix}\cdots
            \begin{pmatrix}
                0 & p_{2^{n}-1}x\\ 1 & 1
            \end{pmatrix}\\
            & = \begin{pmatrix}
                P_{2^{n-1}-3}(x) & P_{2^{n-1}-2}(x)\\ Q_{2^{n-1}-3}(x) & Q_{2^{n-1}-2}(x)
            \end{pmatrix}\begin{pmatrix}
                0 & -x\\ 1 & 1
            \end{pmatrix}
            \begin{pmatrix}
                P_{2^{n-1}-2}^{b}(x) & P_{2^{n-1}-1}^{b}(x)\\ 
                Q_{2^{n-1}-2}^{b}(x) & Q_{2^{n-1}-1}^{b}(x)
            \end{pmatrix}.
        \end{align*}
        Now we obtain that for \(\ell=1\) or \(2\), 
       \begin{equation}\label{eq:pqb-temp}
        \left\{
            \begin{aligned}
                P^{b}_{2^{n}-\ell}(x) & = P^{b}_{2^{n-1}-\ell}(x)P_{2^{n-1}-2}(x) +     Q^{b}_{2^{n-1}-\ell}(x)\left(P_{2^{n-1}-2}(x) -xP_{2^{n-1}-3}(x)    \right),\\
                Q^{b}_{2^{n}-\ell}(x) & = P^{b}_{2^{n-1}-\ell}(x)Q_{2^{n-1}-2}(x) +     Q^{b}_{2^{n-1}-\ell}(x)\left(Q_{2^{n-1}-2}(x)-xQ_{2^{n-1}-3}(x)    \right).
            \end{aligned}
        \right.
       \end{equation}
        Since \(p_{2^{n}-1}=1\) for all \(n\ge 1\), we see  
        \begin{align*}
            P_{2^{n-1}-1}(x) & =P_{2^{n-1}-2}(x)+p_{2^{n-1}-1}xP_{2^{n-1}-3}(x)\\
            & = P_{2^{n-1}-2}(x)+xP_{2^{n-1}-3}(x).
        \end{align*}
        So 
        \begin{equation}
            xP_{2^{n-1}-3}(x)=P_{2^{n-1}-1}(x)-P_{2^{n-1}-2}(x). \label{eq:pn-3}
        \end{equation}
        Similarly, we have 
        \begin{equation}
            xQ_{2^{n-1}-3}(x)=Q_{2^{n-1}-1}(x)-Q_{2^{n-1}-2}(x). \label{eq:qn-3}
        \end{equation}
        Then the equalities \eqref{eq:pb-1} and \eqref{eq:qb-1} follow from \eqref{eq:pqb-temp}, \eqref{eq:pn-3} and \eqref{eq:qn-3}.
    \end{proof}

    To obtain the Stieltjes continued fraction \(\mathrm{Stiel}_{\mathbf{p}}(x)\), we need the following subsequence of convergents.
    \begin{lemma}\label{lem:pq-rec}
        For all \(n\ge 4\), we have 
        \begin{enumerate}[label=(\arabic*)]
            \item \(Q_{2^{n}-2}(x) \equiv_{4} 1 + 2(x+x^{2}+x^{3}+x^{4}) + 2(1+x+x^{2})S_{n-2}(x) + 2T_{n-2}(x),\)
            \item \( Q_{2^{n}-1}(x) \equiv_{4} 1+(1+2x)S_{n-2}(x) + 3x^{2^{n-1}},\)
            \item \( P_{2^{n}-2}(x) \equiv_{4} 2xS_{n-2}(x) + S_{n-1}(x),\)
            \item \( P_{2^{n}-1}(x) \equiv_{4} (3x+2x^{5}) + 2(x+x^{3})S_{n-2}(x) + 2S_{n-1}(x) + 2(1+x)T_{n-2}(x),\)
            \item \(Q_{2^{n}-2}^{b}(x) \equiv_{4} 1+2(x+x^{2}+x^{3}+x^{4}) + 2(x+x^{2})S_{n-2}(x) + 2T_{n-2}(x),\)
            \item \(Q_{2^{n}-1}^{b}(x) \equiv_{4} (1+2x) + 2xS_{n-2}(x) + S_{n-1}(x),\)
            \item \(P_{2^{n}-2}^{b}(x) \equiv_{4} 2x + 2(1+x)S_{n-2}(x) + S_{n-1}(x),\)
            \item \(P_{2^{n}-1}^{b}(x) \equiv_{4} (3x+2x^{5}) + 2(1+x^{3})S_{n-2}(x) + 2(1+x)T_{n-2}(x).\)
        \end{enumerate}
    \end{lemma}
    \begin{proof}
        One can verify the case for \(n=4\). We assume that the above equalities hold for \(n-1\). Now we prove them for \(n\). By the inductive hypothesis, 
        \begin{align*}
            Q_{2^{n-1}-2}(x) & \equiv_{4} 1 + 2(x+x^{2}+x^{3}+x^{4}) + 2(1+x+x^{2})S_{n-3}(x) + 2T_{n-3}(x),\\
            Q_{2^{n-1}-1}(x) & \equiv_{4} 1+(1+2x)S_{n-3}(x) + 3x^{2^{n-2}},\\
            P_{2^{n-1}-2}(x) & \equiv_{4} 2xS_{n-3}(x) + S_{n-2}(x),\\
            P_{2^{n-1}-1}(x)  & \equiv_{4} (3x+2x^{5}) + 2(x+x^{3})S_{n-3}(x) + 2S_{n-2}(x) + 2(1+x)T_{n-3}(x),\\
            Q_{2^{n-1}-2}^{b}(x) & \equiv_{4} 1+2(x+x^{2}+x^{3}+x^{4}) + 2(x+x^{2})S_{n-3}(x) + 2T_{n-3}(x),\\
            Q_{2^{n-1}-1}^{b}(x) & \equiv_{4} (1+2x) + 2xS_{n-3}(x) + S_{n-2}(x),\\
            P_{2^{n-1}-2}^{b}(x) & \equiv_{4} 2x + 2(1+x)S_{n-3}(x) + S_{n-2}(x),\\
            P_{2^{n-1}-1}^{b}(x) & \equiv_{4} (3x+2x^{5}) + 2(1+x^{3})S_{n-3}(x) + 2(1+x)T_{n-3}(x).
        \end{align*}

        (1) It follows from \eqref{eq:q-1} that 
        \begin{align*}
            Q_{2^{n}-2}(x) & = Q_{2^{n-1}-2}(x)P^{b}_{2^{n-1}-2}(x) + Q_{2^{n-1}-1}(x)Q^{b}_{2^{n-1}-2}(x),\\
            & \equiv_{4} 1 + 2(x^{2}+x^{3}+x^{4}) + 2(x+x^{2})S_{n-3}(x) + 2S_{n-3}(x)S_{n-2}(x) + 2T_{n-3}(x)\\
            & \equiv_{4} 1 + 2(x^{2}+x^{3}+x^{4}) + 2(x+x^{2})S_{n-3}(x) \\
            & \quad\ + 2(S_{n-1}(x)-x) + 2x^{2^{n-2}}S_{n-2}(x) \tag*{by \eqref{eq:ss}}\\
            & \quad\ + 2T_{n-2}(x) - 2x^{2^{n-2}}(S_{n-3}(x)-x-x^{2})  \tag*{by  \eqref{eq:t}}\\
            & \equiv_{4} 1 + 2(x+x^{2}+x^{3}+x^{4}) + 2(1+x+x^{2})S_{n-2}(x) + 2T_{n-2}(x).
        \end{align*}

        (2) It follows from \eqref{eq:q-1} that 
        \begin{align*}
            Q_{2^{n}-1}(x) & = Q_{2^{n-1}-2}(x)P^{b}_{2^{n-1}-1}(x) + Q_{2^{n-1}-1}(x)Q^{b}_{2^{n-1}-1}(x),\\
            & \equiv_{4} (1+x+2x^{2}+2x^{3}+2x^{4}) + 2x\cdot x^{2^{n-2}} + 3x^{2^{n-1} }\\
            &\quad\ + 2x^{2}S_{n-3}(x) + 2T_{n-3}(x) + S_{n-3}^{2}(x),\\
            & \equiv_{4} 1+(1+2x)S_{n-2}(x)+3x^{2^{n-1}}. \tag*{by \eqref{eq:s2}}
        \end{align*}

        (3) It follows from \eqref{eq:p-1} that 
        \begin{align*}
            P_{2^{n}-2}(x) & = P_{2^{n-1}-2}(x)P^{b}_{2^{n-1}-2}(x) + P_{2^{n-1}-1}(x)Q^{b}_{2^{n-1}-2}(x),\\
            & \equiv_{4} 3x+2(x^{2}+x^{3}+x^{4})+ 2(x+x^{2})S_{n-3}(x) + 2(1+x)S_{n-2}(x)\\
            & \quad\ + 2S_{n-3}(x)S_{n-2}(x) + S_{n-2}^{2}(x) + 2T_{n-3}(x)\\
            & \equiv_{4} 2xS_{n-2}(x) + S_{n-1}(x). \tag*{by \eqref{eq:ss},  \eqref{eq:t}, \eqref{eq:s2}}
        \end{align*}

        (4) It follows from \eqref{eq:p-1} that
        \begin{align*}
            P_{2^{n}-1}(x) & = P_{2^{n-1}-2}(x)P^{b}_{2^{n-1}-1}(x) + P_{2^{n-1}-1}(x)Q^{b}_{2^{n-1}-1}(x),\\
            & \equiv_{4} (3x+2x^{2}+2x^{5}) + 2(x+x^{3})S_{n-3}(x) + 2(1+x)S_{n-2}(x)\\
            &\quad\ + 2(1+x)S_{n-3}(x)S_{n-2}(x) + 2(1+x)T_{n-3}(x) + 2S_{n-2}^{2}(x),\\
            & \equiv_{4} (3x+2x^{2}+2x^{5}) + 2(1+x^{3})S_{n-2}(x) -2(x+x^{3})x^{2^{n-2}}\\
            & \quad\ + 2(1+x)S_{n-1}(x) -2x(1+x) - 2(1+x)x^{2^{n-2}}S_{n-2}(x) \tag*{by \eqref{eq:ss}}\\
            & \quad\ + 2(1+x)T_{n-2}(x) - 2(1+x)x^{2^{n-2}}S_{n-3}(x) + 2(x+x^{3})x^{2^{n-2}} \tag*{by \eqref{eq:t}}\\
            & \quad\ + 2S_{n-1}(x)-2x\tag*{by \eqref{eq:2s2}}\\
            & \equiv_{4} (3x+2x^{5}) + 2(1+x^{3})S_{n-2}(x) + 2xS_{n-1}(x) -2(1+x)x^{2^{n-1}} + 2(1+x)T_{n-2}(x)\\
            & \equiv_{4} (3x+2x^{5}) + 2(x+x^{3})S_{n-2}(x) + 2S_{n-1}(x) + 2(1+x)T_{n-2}(x). 
        \end{align*}

        (5) It follows from \eqref{eq:qb-1} that 
        \begin{align*}
            Q^{b}_{2^{n}-2}(x) & \equiv_{4} Q_{2^{n}-2}(x) +  Q^{b}_{2^{n-1}-2}(x)\left(2Q_{2^{n-1}-2}(x) + 2Q_{2^{n-1}-1}(x)\right)\\
            & \equiv_{4} Q_{2^{n}-2}(x) + 2S_{n-2}(x).
        \end{align*}

        (6) According to \eqref{eq:qb-1},
        \begin{align*}
            Q^{b}_{2^{n}-1}(x) & \equiv_{4} Q_{2^{n}-1}(x)+ Q^{b}_{2^{n-1}-1}(x)\left(2Q_{2^{n-1}-2}(x) +2Q_{2^{n-1}-1}(x)\right)\\
            & \equiv_{4} Q_{2^{n}-1}(x) + 2(1+S_{n-2}(x))S_{n-2}(x)\\
            & \equiv_{4} (1+2x) + 2xS_{n-2}(x) + S_{n-1}(x). \tag*{by \eqref{eq:2s2}}
        \end{align*}

        (7) It follows from \eqref{eq:pb-1} that,
        \begin{align*}
            P^{b}_{2^{n}-2}(x) & \equiv_{4} P_{2^{n}-2}(x) + Q^{b}_{2^{n-1}-2}(x)\left(2P_{2^{n-1}-2}(x) + 2P_{2^{n-1}-1}(x)\right)\\
            & \equiv_{4} P_{2^{n}-2}(x) + 2x + 2S_{n-2}(x).
        \end{align*}
        
        (8) According to \eqref{eq:pb-1}, 
        \begin{align*}
            P^{b}_{2^{n}-1}(x) & \equiv_{4} P_{2^{n}-1}(x) + Q^{b}_{2^{n-1}-1}(x)\left(2P_{2^{n-1}-2}(x) + 2P_{2^{n-1}-1}(x)\right)\\
            & \equiv_{4} P_{2^{n}-1}(x) + 2(1+S_{n-2}(x))(x+S_{n-2}(x))\\
            & \equiv_{4} (3x+2x^{5}) + 2(1+x^{3})S_{n-2}(x) + 2(1+x)T_{n-2}(x). \tag*{by \eqref{eq:2s2}}
        \end{align*}
        By induction, the result holds from the above.
    \end{proof}

    In the following we give the explicit expression of \(\mathrm{Stiel}_{\mathbf{p}}(x)\) from its convergents. As we shall see later, the Stieltjes continued fraction \(\mathrm{Stiel}_{\mathbf{p}}(x)\) is related to the generating function of Catalan numbers. The \(n\)th Catalan number is \(C_{n}=\frac{1}{n+1}\binom{2n}{n}\). It is well known that the generating function \(\phi(x):=\sum_{n\geq 0}C_{n}x^{n} \) of the Catalan numbers satisfies
    \begin{equation}
        \phi(x) = 1 + x\phi^{2}(x);\label{eq:phi}
    \end{equation}
    see for example \cite{Ni11, St99}. Moreover, one has
    \begin{equation}
        \phi(x) = \frac{1-\sqrt{1-4x}}{2x}.\label{eq:gen_catalan}
    \end{equation}
    The next lemma gives explicit values of the Catalan numbers modulo \(4\); see \cite[Theorem 2.3]{ELY08}. 
    \begin{lemma}\label{lem:catalan}
        Let \(C_n\) be the \(n\)th Catalan number. Then 
        \[C_n \equiv_{4} 
        \begin{cases}
            1, & \text{if } n=2^{a}-1 \text{ for some }a\ge 0;\\
            2, & \text{if } n=2^{b}+2^{a}-1 \text{ for some }b>a\ge 0;\\
            0, & \text{otherwise.} 
        \end{cases}\]
    \end{lemma}

    \begin{proposition}\label{prop:algebraic}
        The infinite Stieltjes continued fraction \(\mathrm{Stiel}_{\mathbf{p}}(x)\) defined by the paperfolding sequence \(\mathbf{p}\) is congruent modulo $4$ to an algebraic series in \(\mathbb{Z}[[x]]\). Namely, \(\mathrm{Stiel}_{\mathbf{p}}(x)  \equiv_4 2x + (3x+2x^{3})\phi(x).\)
    \end{proposition}

    \begin{proof}
        From Lemma \ref{lem:pq-rec}, we have \(Q_{2^{n}-2}^{2}(x)\equiv_4 1\). Using this fact, we obtain that
        \begin{align}
            \mathrm{Stiel}_{\mathbf{p}}(x) & \equiv_4  \lim_{n\to\infty}\frac{P_    {2^{n}-2}(x)}{Q_{2^{n}-2}(x)}\nonumber\\
            & \equiv_4 \lim_{n\to\infty}\frac{P_{2^{n}-2}(x)Q_{2^{n}-2}(x)}{Q_{2^    {n}-2}^{2}(x)}\nonumber\\
            & \equiv_4 \lim_{n\to\infty}P_{2^{n}-2}(x)Q_{2^{n}-2}(x)\quad\qquad \text{since}~Q_{2^{n}-2}(0)=1\nonumber\\
            & \equiv_4\lim_{n\to\infty}\left(2xS_{n-2}(x) + S_{n-1}(x) + 2(x+x^{2}+x^{3}+x^{4})S_{n-1}(x)\right.\nonumber\\
            & \qquad\qquad \left. +2(1+x+x^{2})S_{n-2}(x)S_{n-1}(x)+2T_{n-2}(x)S_{n-1}(x)\right)\nonumber\\
            & \equiv_4 (1+2x^{2}+2x^{3}+2x^{4})S_{\infty}(x) + 2(1+x+x^{2})S^{2}_{\infty}(x) + 2T_{\infty}(x)S_{\infty}(x).\label{eq:stiel_p_1}
        \end{align}
        We need to calculate  \(2T_{\infty}(x)~(\mathrm{mod}~4)\), \(S^{2}_{\infty}(x)~(\mathrm{mod}~4)\) and \(S_{\infty}(x)~(\mathrm{mod}~4)\). By Lemma \ref{lem:catalan}, we have
        \begin{align}
            \phi(x) & \equiv_4 \sum_{a=0}^{+\infty}x^{2^{a}-1}     + 2\sum_{b=1}^{+\infty}\sum_{a=0}^{b-1}x^{2^{b}+2^{a}-1}\nonumber\\
            & \equiv_4 x^{-1}S_{\infty}(x) + 2x^{-1}\left(\sum_{b=3}^{+\infty}    \sum_{a=0}^{b-1}x^{2^{b}+2^{a}}+(x^{3}+x^{5}+x^{6})\right)\nonumber\\
            & \equiv_4 x^{-1}S_{\infty}(x) + 2x^{-1}\left(\sum_{b=3}^{+\infty}    \sum_{a=2}^{b-1}x^{2^{b}+2^{a}}+ (x+x^{2})\sum_{b=3}^{+\infty}x^{2^    {b}}+(x^{3}+x^{5}+x^{6})\right)\nonumber\\
            & \equiv_4 x^{-1}S_{\infty}(x) + 2x^{-1}T_{\infty}(x)+2(1+x)S_    {\infty}(x)+2(x+x^{2}+x^{3}). \label{eq:ct}
        \end{align}
        Moreover, Eq. \eqref{eq:ct} yields that \(S_{\infty}(x)\equiv_{2} x\phi(x)\). So
        \begin{align}
            S_{\infty}^{2}(x) & \equiv_4 \left(x\phi(x)\right)^{2}. \label{eq:s2_mod_4}
        \end{align} 
    To evaluate \(S_{\infty}(x)~(\mathrm{mod}~4)\), we use \eqref{eq:s2_mod_4}. Then  
    \begin{align*}
        S_{\infty}(x)-x-x^{2} & = S_{\infty}(x^{4}) \equiv_4 S_{\infty}^{4}(x) \equiv_4 \left(x\phi(x)\right)^{4}
    \end{align*}
    which implies
    \begin{align}
        S_{\infty}(x) & \equiv_4 \left(x\phi(x)\right)^{4}+x+x^{2}.\label{eq:s_mod_4}
    \end{align}
    Now, combining \eqref{eq:stiel_p_1} and \eqref{eq:ct}, we have 
    \begin{align*}
        \mathrm{Stiel}_{\mathbf{p}}(x) & \equiv_4 S_{\infty}(x) + S_{\infty}^{2}(x) + xS_{\infty}(x)\phi(x)\\
        & \equiv_4 (x+x^{2}) + (x^{2}+x^{3})\phi(x) + \left(x\phi(x)\right)^{2} +\left(x\phi(x)\right)^{4} + \left(x\phi(x)\right)^{5} \tag*{by \eqref{eq:s2_mod_4} and \eqref{eq:s_mod_4}}\\
        & \equiv_4 2x + (3x+2x^{3})\phi(x). \tag*{by \eqref{eq:phi}}
    \end{align*}
    \end{proof}
    
    Now we have seen that \(\mathrm{Stiel}_{\mathbf{p}}(x)\) is congruent modulo $4$ to an algebraic series in \(\mathbb{Z}[[x]]\). Combining with Denef-Lipshitz's result, we have the automaticity of \(\mathrm{Stiel}_{\mathbf{p}}\) \(\mathrm{mod}\ 4\).

   \begin{proof}[Proof of Theorem \ref{thm:01}]
        The first part follows from Proposition \ref{prop:algebraic}. Then using Theorem \ref{thm:ckmr}, we see that \((a_{n})_{n\ge 0}\) is \(2\)-automatic.
    \end{proof}

    \section{Rudin-Shapiro sequence and its Stieltjes continued fraction}\label{sec:rudin}
    Let \(r\) be the Rudin-Shapiro sequence over the alphabet \(\left\{-1,1\right\}\). Recall that the sequence can also be generated by the substitution \[\sigma_{rs}: a\to ab,\ b\to ac,\ c\to db,\ d\to dc\] and then the projection \[\rho_{rs}: a\to 1,\, b\to 1,\ c\to -1,\ d\to -1.\]
    The next observation is useful while deducing the recurrence relation of the convergents.   
    \begin{lemma}\label{lem:rs-1}
        Let \(\iota\) be the coding on \(\{-1,1\}\) which maps \(1\to -1\) and \(-1 \to 1\). Then for all \(n\ge 0\),
        \[\rho_{rs}(\sigma_{rs}^{n}(a))=\iota(\rho_{rs}(\sigma_{rs}^{n}(d)))\quad\text{and}\quad\rho_{rs}(\sigma_{rs}^{n}(b))=\iota(\rho_{rs}(\sigma_{rs}^{n}(c))).\] 
    \end{lemma} 
    \begin{proof}
        It is clear that the result holds for \(n=0\). Now suppose the result holds for all \(n\leq m\). Since 
        \begin{align*}
            \rho_{rs}(\sigma_{rs}^{m+1}(a)) & = \rho_{rs}(\sigma_{rs}^{m}(a))\rho_{rs}(\sigma_{rs}^{m}(b))\\
            & = \iota(\rho_{rs}(\sigma_{rs}^{m}(d)))\iota(\rho_{rs}(\sigma_{rs}^{m}(c)))\qquad\text{by the induction hypothesis}\\
            & = \iota(\rho_{rs}(\sigma_{rs}^{m}(d))\sigma_{rs}^{m}(c)))\\
            & = \iota(\rho_{rs}(\sigma_{rs}^{m+1}(d)))
        \end{align*}
        which shows that the first equality in the statement holds for \(n=m+1\). The validity of the second equality for \(n=m+1\) follows in the same way. This proves the result.
    \end{proof}


    To reduce the number of new notations, we redefine \(P, Q, P^{b}, Q^{b}\). Let \(P_{n}(x)/Q_{n}(x)\) be the \(n\)th convergent of \(\mathrm{Stiel}_{\mathbf{r}}(x)\). Namely, \(P_{n}(x)/Q_{n}(x):=\mathrm{Stiel}_{\mathbf{r}|_n}(x)\). In addition, for all \(n\ge 2\), we define \(P_{2^{n}-1}^{b}(x)/Q_{2^{n}-1}^{b}(x):=\text{Stiel}_{\rho(\sigma^{n}(b))}(x)\). Since \(\sigma_{rs}^{n}(a)=\sigma_{rs}^{n-1}(a)\sigma_{rs}^{n-1}(b)\), for all \(n\ge 2\),
    \begin{equation*}
        \begin{pmatrix}
            P_{2^{n}-2}^{b}(x) & P_{2^{n}-1}^{b}(x)\\ 
            Q_{2^{n}-2}^{b}(x) & Q_{2^{n}-1}^{b}(x)
        \end{pmatrix} := 
        \begin{pmatrix}
            0 & r_{2^{n}}x\\ 1 & 1
        \end{pmatrix}
        \begin{pmatrix}
            0 & r_{2^{n}+1}x\\ 1 & 1
        \end{pmatrix}\dots
        \begin{pmatrix}
            0 & r_{2^{n+1}-1}x\\ 1 & 1
        \end{pmatrix}.
    \end{equation*}

    The relation of \(P, Q, P^{b}, Q^{b}\) are formalized in the following lemma.
    \begin{lemma}
        For all \(n\ge 2\) and \(\ell=1\) or \(2\),
        \begin{align}
            P_{2^{n}-\ell}(x) & = P_{2^{n-1}-2}(x)P^{b}_{2^{n-1}-\ell}(x) + P_{2^{n-1}-1}(x)Q^{b}_{2^{n-1}-\ell}(x),\label{eq:rs-p-1}\\
            Q_{2^{n}-\ell}(x) & = Q_{2^{n-1}-2}(x)P^{b}_{2^{n-1}-\ell}(x) + Q_{2^{n-1}-1}(x)Q^{b}_{2^{n-1}-\ell}(x),\label{eq:rs-q-1}\\
            P^{b}_{2^{n}-\ell}(x) & = P_{2^{n-1}-2}(x)P^{b}_{2^{n-1}-\ell}(-x) + P_{2^{n-1}-1}(x)Q^{b}_{2^{n-1}-\ell}(-x),\label{eq:rs-pb-1}\\
            Q^{b}_{2^{n}-\ell}(x) & = Q_{2^{n-1}-2}(x)P^{b}_{2^{n-1}-\ell}(-x) + Q_{2^{n-1}-1}(x)Q^{b}_{2^{n-1}-\ell}(-x).\label{eq:rs-qb-1}        \end{align}
    \end{lemma}
    \begin{proof}
        The equalities \eqref{eq:rs-p-1} and \eqref{eq:rs-q-1} follows from the fact that 
        \begin{align*}
            \begin{pmatrix}
                P_{2^{n}-2}(x) & P_{2^{n}-1}(x)\\ Q_{2^{n}-2}(x) & Q_{2^{n}-1}(x)
            \end{pmatrix}  & = 
            \begin{pmatrix}
                0 & r_{0}x\\ 1 & 1
            \end{pmatrix}
            \begin{pmatrix}
                0 & r_{1}x\\ 1 & 1
            \end{pmatrix}\dots
            \begin{pmatrix}
                0 & r_{2^{n}-1}x\\ 1 & 1
            \end{pmatrix}\\
            & =  \begin{pmatrix}
                0 & r_{0}x\\ 1 & 1
            \end{pmatrix}\dots
            \begin{pmatrix}
                0 & r_{2^{n-1}-1}x\\ 1 & 1
            \end{pmatrix}
            \begin{pmatrix}
                0 & r_{2^{n-1}}x\\ 1 & 1
            \end{pmatrix}\dots
            \begin{pmatrix}
                0 & r_{2^{n}-1}x\\ 1 & 1
            \end{pmatrix}\\
            & = \begin{pmatrix}
                P_{2^{n-1}-2}(x) & P_{2^{n-1}-1}(x)\\ Q_{2^{n-1}-2}(x) & Q_{2^{n-1}-1}(x)
            \end{pmatrix}\begin{pmatrix}
                P_{2^{n-1}-2}^{b}(x) & P_{2^{n-1}-1}^{b}(x)\\ 
                Q_{2^{n-1}-2}^{b}(x) & Q_{2^{n-1}-1}^{b}(x)
            \end{pmatrix}.
        \end{align*}
        Since \(\sigma^{n}(b)=\sigma^{n-1}(a)\sigma^{n-1}(c)\), we have 
        \begin{align*}
            r_{2^{n}+2^{n-1}}r_{2^{n}+2^{n-1}+1}\dots r_{2^{n+1}-1} & =\rho(\sigma^{n-1}(c))\\
            & =\iota(\sigma^{n-1}(b))\qquad\text{(by Lemma \ref{lem:rs-1})}\\
            & = (-r_{2^{n-1}})(-r_{2^{n-1}+1})\dots(-r_{2^{n}-1}).
        \end{align*}
        Therefore, 
        \begin{align*}
            \begin{pmatrix}
                P_{2^{n}-2}^{b}(x) & P_{2^{n}-1}^{b}(x)\\ 
                Q_{2^{n}-2}^{b}(x) & Q_{2^{n}-1}^{b}(x)
            \end{pmatrix} & = 
            \begin{pmatrix}
                0 & r_{2^{n}}x\\ 1 & 1
            \end{pmatrix}
            \begin{pmatrix}
                0 & r_{2^{n}+1}x\\ 1 & 1
            \end{pmatrix}\dots
            \begin{pmatrix}
                0 & r_{2^{n+1}-1}x\\ 1 & 1
            \end{pmatrix}\\
            & = \begin{pmatrix}
                0 & r_{0}x\\ 1 & 1
            \end{pmatrix}
            \begin{pmatrix}
                0 & r_{1}x\\ 1 & 1
            \end{pmatrix}\dots
            \begin{pmatrix}
                0 & r_{2^{n-1}-1}x\\ 1 & 1
            \end{pmatrix}\\
            & \qquad\cdot
            \begin{pmatrix}
                0 & -r_{2^{n-1}}x\\ 1 & 1
            \end{pmatrix}
            \begin{pmatrix}
                0 & -r_{2^{n-1}+1}x\\ 1 & 1
            \end{pmatrix}\dots
            \begin{pmatrix}
                0 & -r_{2^{n}-1}x\\ 1 & 1
            \end{pmatrix}\\
            & = \begin{pmatrix}
                P_{2^{n-1}-2}(x) & P_{2^{n-1}-1}(x)\\ 
                Q_{2^{n-1}-2}(x) & Q_{2^{n-1}-1}(x)
            \end{pmatrix}\begin{pmatrix}
                P_{2^{n-1}-2}^{b}(-x) & P_{2^{n-1}-1}^{b}(-x)\\ 
                Q_{2^{n-1}-2}^{b}(-x) & Q_{2^{n-1}-1}^{b}(-x)
            \end{pmatrix}
        \end{align*}
        which proves \eqref{eq:rs-pb-1} and \eqref{eq:rs-qb-1}.
    \end{proof}

    To obtain the Stieltjes continued fraction \(\mathrm{Stiel}_{\mathbf{r}}(x)\), we need at least one  subsequence of convergents. 
    \begin{lemma}\label{lem:7}
        For all \(j\ge 2\),
        \begin{enumerate}[label=(\arabic*)]
            \item \(Q_{2^{2j}-2}(x)  \equiv_{4} 1+2x+2(1+x)S_{2j-2}(x).\)
            \item \(Q_{2^{2j+1}-2}(x) \equiv_{4} 1+ 2(1+x)S_{2j-1}(x),\)
            \item \(Q_{2^{2j}-1}(x) \equiv_{4} 1+2x^{2}+2x^{5}+2xS_{j-2}^{o}(x) + (3+2x^{3})S_{2j-2}(x) + 2(1+x)T_{2j-2}(x) + x^{2^{2j-1}},\)
            \item \(Q_{2^{2j+1}-1}(x)  \equiv_{4} 1+2x^{2}+2x^{5}+2xS_{j-1}^{e}(x) + (3+2x^{3})S_{2j-1}(x) + 2(1+x)T_{2j-1}(x) + x^{2^{2j}},\)
            \item \(P_{2^{2j}-2}(x) \equiv_{4} 2x^{2}+2x^{5} + (1+2x^{3})S_{2j-2}(x) + 2(1+x)T_{2j-2}(x) +2xS^{o}_{j-2}(x)+ x^{2^{2j-1}},\)
            \item \(P_{2^{2j+1}-2}(x)  \equiv_{4} 2x^{2}+2x^{5} + (1+2x^{3})S_{2j-1}(x) + 2(1+x)T_{2j-1}(x) +2xS^{e}_{j-1}(x)+ x^{2^{2j}},\)
            \item \(P_{2^{2j}-1}(x) \equiv_{4} x+2x^{2}+2x^{3}+2x^{4}+2x^{5}+2x^{3}S_{2j-2}(x)+2xT_{2j-2}(x)+2xS^{o}_{j-2}(x),\)
            \item \(P_{2^{2j+1}-1}(x) \equiv_{4} x+2x^{3}+2x^{4}+2x^{5}+2x^{3}S_{2j-1}(x)+2xT_{2j-1}(x)+2xS^{e}_{j-1}(x),\)
            \item \(Q_{2^{2j}-2}^{b}(x)\equiv_{4} 1+2(1+x)S_{2j-2}(x),\)
            \item \(Q_{2^{2j+1}-2}^{b}(x)\equiv_{4} 1+2x+2(1+x)S_{2j-1}(x),\)
            \item \(Q_{2^{2j}-1}^{b}(x) \equiv_{4} 1+2x^{2}+2x^{5}+(3+2x^{3})S_{2j-2}(x)+(2+2x)T_{2j-2}(x) + 2xS_{j-1}^{e}(x) + x^{2^{2j-1}},\)
            \item \(Q_{2^{2j+1}-1}^{b}(x) \equiv_{4} 1+2x^{2}+2x^{5}+(3+2x^{3})S_{2j-1}(x)+(2+2x)T_{2j-1}(x) + 2xS_{j-1}^{o}(x) + x^{2^{2j}},\)
            \item \(P_{2^{2j}-2}^{b}(x) \equiv_{4} 2x^{2}+2x^{5} + (1+2x^{3})S_{2j-2}(x)+2(1+x)T_{2j-2}(x)+2xS_{j-1}^{e}(x)+x^{2^{2j-1}},\)
            \item \(P_{2^{2j+1}-2}^{b}(x) \equiv_{4} 2x^{2}+2x^{5} + (1+2x^{3})S_{2j-1}(x)+2(1+x)T_{2j-1}(x)+2xS_{j-1}^{o}(x)+x^{2^{2j}},\)
            \item \(P_{2^{2j}-1}^{b}(x) \equiv_{4} x+2x^{3}+2x^{4}+2x^{5}+2xS_{j-1}^{e}(x)+2x^{3}S_{2j-2}(x)+2xT_{2j-2}(x),\)
            \item \(P_{2^{2j+1}-1}^{b}(x) \equiv_{4} x+2x^{2}+2x^{3}+2x^{4}+2x^{5}+2xS_{j-1}^{o}(x)+2x^{3}S_{2j-1}(x)+2xT_{2j-1}(x)\).
        \end{enumerate}
    \end{lemma}
    \begin{proof}
        The initial values for \(j=2\) can be calculated directly. Now we suppose the result holds for \(j\). We verify it for \(j+1\).
        
        (1) By \eqref{eq:rs-q-1} and the induction hypothesis (2) and (10), we have 
        \begin{align*}
            Q_{2^{2(j+1)}-2}(x) & \equiv_4 Q_{2^{2j+1}-2}(x)P^{b}_{2^{2j+1}-2}(x) + Q_{2^{2j+1}-1}(x)Q^{b}_{2^{2j+1}-2}(x)\\
            & \equiv_{4} P^{b}_{2^{2j+1}-2}(x) + Q_{2^{2j+1}-1}(x)+2x+2S_{2j-1}(x)+2x^{2^{2j}+1}\\
            & \equiv_{4} \left(1+2xS_{2j-1}(x) + 2x^{2^{2j}}\right) + 2x+2S_{2j-1}(x)+2x^{2^{2j}+1}\\
            & \equiv_{4} 1+2x+2(1+x)S_{2j}(x)
        \end{align*}
        where the next to the last equality follows from the inductive hypothesis (4) and (14).

        (3) By \eqref{eq:rs-q-1}, 
        \begin{align*}
            Q_{2^{2(j+1)}-1}(x) & \equiv_4 Q_{2^{2j+1}-2}(x)P^{b}_{2^{2j+1}-1}(x) + Q_{2^{2j+1}-1}(x)Q^{b}_{2^{2j+1}-1}(x)\\
            & \equiv_{4} P^{b}_{2^{2j+1}-1}(x)+2(x+x^{2})S_{2j-1}(x) + (1+S_{2j}(x))^{2} + 2x(1+S_{2j}(x))S_{2j-1}(x)\\
            & \equiv_{4}  P^{b}_{2^{2j+1}-1}(x)+2x^{2}S_{2j-1}(x)+1+2S_{2j}(x)+S^{2}_{2j}(x)+2xS_{2j}(x)S_{2j-1}(x) \\
            & \equiv_{4}  P^{b}_{2^{2j+1}-1}(x)+2x^{2}S_{2j-1}(x)+1+2S_{2j}(x)\\
            & \qquad +(3x+2x^{2}+2x^{3}+2x^{4}) + 2(x+x^{2})S_{2j}(x) + S_{2j+1}(x) + 2T_{2j}(x)\\
            & \qquad + 2x(S_{2j+1}(x)-x)+2xx^{2^{2j}}S_{2j}(x) \tag*{by \eqref{eq:s2} and \eqref{eq:ss}}\\
            & \equiv_{4} P^{b}_{2^{2j+1}-1}(x) + (1+3x+2x^{3}+2x^{4}) + 3S_{2j}(x) + x^{2^{2j+1}} + 2T_{2j}(x)\\
            & \qquad + 2xx^{2j}S_{2j-1}(x) + 2x^{2}x^{2^{2j}}\\
            & \equiv_{4} (1+2x^{2}+2x^{5}) + 2xS^{o}_{j-1}(x) + (3+2x^{3})S_{2j}(x) + x^{2^{2j+1}} + 2T_{2j}(x) \\
            & \qquad + 2xT_{2j-1}(x)+ 2xx^{2j}\left(S_{2j-1}(x) + x+x^{2}\right)\tag*{by the hypothesis (16)} \\
            & \equiv_{4} (1+2x^{2}+2x^{5}) + 2xS^{o}_{j-1}(x) + (3+2x^{3})S_{2j}(x) + x^{2^{2j+1}} + 2(1+x)T_{2j}(x). \tag*{by \eqref{eq:t}}
        \end{align*}

        (5) By \eqref{eq:rs-p-1} and the inductive hypothesis (6), (8), (10) and (14), we obtain that
        \begin{align*}
            P_{2^{2(j+1)}-2}(x) & = P_{2^{2j+1}-2}(x)P^{b}_{2^{2j+1}-2}(x) + P_{2^{2j+1}-1}(x)Q^{b}_{2^{2j+1}-2}(x)\\
            & \equiv_{4} (1+2x)S_{2j-1}^{2}(x)+2(1+x)x^{2^{2j}}S_{2j-1}(x)+x^{2^{2j+1}}\\
            & \qquad +P_{2^{2j+1}-1}(x) + 2x^{2} + 2(x+x^{2})S_{2j-1}(x)\\
            & \equiv_{4} 2x^{2} + 2x^{5} + (1+2x^{3})S_{2j}(x) + 2(1+x)T_{2j}(x) + 2xS_{j-1}^{o}(x) + x^{2^{2j+1}}.
            \tag*{by \eqref{eq:s2} and \eqref{eq:t}}
        \end{align*}

        (7) From the inductive hypothesis, we have 
        \begin{equation}\label{eq:lem:7-1}\left\{
            \begin{aligned}
                P_{2^{2j+1}-2}(x) + Q^{b}_{2^{2j+1}-1}(x) & \equiv_{4} 1+2xS_{2j-1}(x) + 2x^{2^{2j}},\\
            P_{2^{2j+1}-1}(x) - P_{2^{2j+1}-1}^{b}(x) & \equiv_{4} 2x^{2}+2xS_{2j-1}(x).
            \end{aligned}\right.
        \end{equation}
        It follows from \eqref{eq:rs-p-1} and \eqref{eq:lem:7-1} that  
        \begin{align*}
            P_{2^{2(j+1)}-1}(x) & = P_{2^{2j+1}-2}(x)P^{b}_{2^{2j+1}-1}(x) + P_{2^{2j+1}-1}(x)Q^{b}_{2^{2j+1}-1}(x)\\
            & \equiv_{4} (1+2xS_{2j-1}(x) + 2x^{2^{2j}})P^{b}_{2^{2j+1}-1}(x) + (2x^{2}+2xS_{2j-1}(x))Q^{b}_{2^{2j+1}-1}(x)\\
            & \equiv_{4} P^{b}_{2^{2j+1}-1}(x) + 2x^{2}S_{2j-1}(x) + 2x^{2^{2j}+1} \\
            & \qquad + (2x^{2}+2xS_{2j-1}(x))(1+S_{2j-1}(x)+x^{2^{2j}}) \\
            & \equiv_{4} x+2x^{2}+2x^{3}+2x^{4}+2x^{5}+2x^{3}S_{2j}(x)+2xT_{2j}(x)+2xS^{o}_{j-1}(x).
            \tag*{by \eqref{eq:2s2} and \eqref{eq:t}}
        \end{align*}

        (9) By \eqref{eq:rs-qb-1} and the inductive hypothesis,
        \begin{align*}
            Q^{b}_{2^{2(j+1)}-2}(x) & = Q_{2^{2j+1}-2}(x)P^{b}_{2^{2j+1}-2}(-x) + Q_{2^{2j+1}-1}(x)Q^{b}_{2^{2j+1}-2}(-x)\\
            & \equiv_{4} P^{b}_{2^{2j+1}-2}(-x) + 2(1+x)S_{2j-1}(x)S_{2j}(-x)\\
            &\qquad + Q_{2^{2j+1}-1}(x)+ (-2x+2(1-x)S_{2j-1}(-x))(1+S_{2j}(x))\\
            & \equiv_{4} \left(P^{b}_{2^{2j+1}-2}(-x) + Q_{2^{2j+1}-1}(x)\right) + 2x+2S_{2j-1}(x)  + 2x^{2^{2j}+1}\\
            & \equiv_{4} \left(1+2x+2xS_{2j-1}(x)+2x^{2^{2j}}\right)+ 2x+2S_{2j-1}(x)  + 2x^{2^{2j}+1}\\
            & \equiv_{4} 1 + 2(1+x)S_{2j}(x).
        \end{align*}

        (11) By \eqref{eq:rs-qb-1} and the inductive hypothesis,
        \begin{align*}
            Q^{b}_{2^{2(j+1)}-1}(x) & = Q_{2^{2j+1}-2}(x)P^{b}_{2^{2j+1}-1}(-x) + Q_{2^{2j+1}-1}(x)Q^{b}_{2^{2j+1}-1}(-x)\\
            & \equiv_{4} P^{b}_{2^{2j+1}-1}(-x) + 2x(1+x)S_{2j-1}(x)\\
            &\qquad + \left(1+S_{2j}(x)\right)^{2} + \left(1+S_{2j}(x)\right)\left(2x+2xS_{2j-1}(x)\right)\\
            & \equiv_{4} (1+x+2x^{2}+2x^{3}+2x^{4}+2x^{5}) + 2S_{2j}(x)+2xS_{j}^{e}(x)\\
            & \qquad +2(x^{2}+x^{3})S_{2j-1}(x)+ S_{2j}^{2}(x) + 2xS_{2j-1}(x)S_{2j}(x) + 2xT_{2j-1}(x)\\
            & \equiv_{4} (1+x+2x^{3}+2x^{4}+2x^{5}) + 2(1+x+x^{2}+x^{3})S_{2j}(x)+ 2xS_{j}^{e}(x) \\
            & \qquad + S_{2j}^{2}(x) +  2xT_{2j}(x)
            \tag*{by \eqref{eq:s2} and \eqref{eq:t}}\\
            & \equiv_{4} 1+2x^{2}+2x^{5} + 2(1+x^{3})S_{2j}(x)+ 2xS_{j}^{e}(x) + 2(1+x)T_{2j}(x) + S_{2j+1}(x).
            \tag*{by \eqref{eq:s2}}
        \end{align*}

        (13) By \eqref{eq:rs-pb-1} and the inductive hypothesis,
        \begin{align*}
            P^{b}_{2^{2(j+1)}-2}(x) & = P_{2^{2j+1}-2}(x)P^{b}_{2^{2j+1}-2}(-x) + P_{2^{2j+1}-1}(x)Q^{b}_{2^{2j+1}-2}(-x),\\
            & \equiv_{4} S_{2j}^{2}(x) +2x[1+S_{2j-1}(x)]S_{2j}(x) + (x+2x^{2}+2x^{3}+2x^{4}+2x^{5})\\
            & \qquad 2(x + x^{2} + x^{3})S_{2j-1}(x) +2xT_{2j-1}(x) + 2xS^{e}_{j-1}(x)\\
            & \equiv_{4} 2x^{5} + 2(x+x^{2})x^{2^{2j}} +2x^{3}S_{2j-1}(x) +S_{2j+1}(x)+2T_{2j}(x)\\
            & \qquad +2x[1+S_{2j-1}(x)]S_{2j}(x)+2xT_{2j-1}(x) + 2xS^{e}_{j-1}(x)
            \tag*{by \eqref{eq:s2}}\\
            & \equiv_{4} 2x^{5} + 2x^{2}x^{2^{2j}} +2x^{3}S_{2j-1}(x) +S_{2j+1}(x)+2T_{2j}(x)+2xS_{2j}(x) + 2xS^{e}_{j}(x)\\
            & \qquad +2xT_{2j}(x) + 2xS_{2j}(x) +2x^{2} +2(x^{2}+x^{3})x^{2^{2j}} \tag*{by \eqref{eq:ss} and \eqref{eq:t}}\\
            & \equiv_{4} 2x^{2} + 2x^{5} + 2(1+x)T_{2j}(x) +2x^{3}S_{2j}(x) +S_{2j+1}(x) + 2xS_{j}^{e}(x).
        \end{align*}

        (15) By \eqref{eq:rs-pb-1} and the inductive hypothesis,
        \begin{align*}
            P^{b}_{2^{2(j+1)}-1}(x) & = P_{2^{2j+1}-2}(x)P^{b}_{2^{2j+1}-1}(-x) + P_{2^{2j+1}-1}(x)Q^{b}_{2^{2j+1}-1}(-x),\\
            & \equiv_{4} xP_{2^{2j+1}-2}(x) + S_{2j}(x)\left(P^{b}_{2^{2j+1}-1}(x)+x\right) \\
            & \qquad + \left(1+S_{2j}(x)\right)P_{2^{2j+1}-1}(x) + x\left(Q^{b}_{2^{2j+1}-1}(x)-1+2x-S_{2j}(x)\right)\\
            & \equiv_{4} 2xS_{2j-1}(x)+2x^{2}x^{2^{2j}}+2x^{2}+2xS_{2j-1}(x)S_{2j}(x)+P_{2^{2j+1}-1}(x)\\
            & \equiv_{4} x+2x^{2}+2x^{3}+2x^{4}+2x^{5}+2xS_{2j-1}(x)+2xS^{e}_{j-1}(x)+2x^{3}S_{2j-1}(x)+2x^{2}x^{2^{2j}}\\
            & \qquad+2xS_{2j-1}(x)S_{2j}(x)+2xT_{2j-1}(x) \\
            & \equiv_{4} x+2x^{2}+2x^{3}+2x^{4}+2x^{5}+2xS_{2j-1}(x)+2xS^{e}_{j-1}(x)+2x^{3}S_{2j-1}(x)+2x^{2}x^{2^{2j}}\\
            & \qquad +2xS_{2j}(x)-2x^{2} + 2xT_{2j}(x) - 2(x^{2}+x^{3})x^{2^{2j}}
            \tag*{by \eqref{eq:ss} and \eqref{eq:t}}\\
            & \equiv_{4} x+2x^{3}+2x^{4}+2x^{5}+2x^{3}S_{2j}(x)+2xS^{e}_{j}(x)+ 2xT_{2j}(x).
        \end{align*}
        
        The above show that the odd numbered equations hold for \(h=j+1\). Based on these results, in the following, we deal with the even numbered equations.

        (2) By \eqref{eq:rs-q-1} and the equalities (1), (3), (9) and (13) for \(j+1\), we obtain that 
        \begin{align*}
            Q_{2^{2j+3}-2}(x) & = Q_{2^{2j+2}-2}(x)P^{b}_{2^{2j+2}-2}(x) + Q_{2^{2j+2}-1}(x)Q^{b}_{2^{2j+2}-2}(x)\\
            & \equiv_{4} P^{b}_{2^{2j+2}-2}(x)+ 2xS_{2j+1}(x) + 2(1+x)S_{2j}(x)S_{2j+1}(x)\\
            & \qquad +Q_{2^{2j+2}-1}(x) +2(1+x)S_{2j}(x)+2(1+x)S_{2j}(x)S_{2j+1}(x)\\
            &\equiv_{4}  1+2(1+x)S_{2j+1}(x).
        \end{align*}

        (4) By \eqref{eq:rs-q-1} and the equalities (1), (3), (11) and (15) for \(j+1\), we obtain that 
        \begin{align*}
            Q_{2^{2j+3}-1}(x) & = Q_{2^{2j+2}-2}(x)P^{b}_{2^{2j+2}-1}(x) + Q_{2^{2j+2}-1}(x)Q^{b}_{2^{2j+2}-1}(x)\\
            & \equiv_{4} P^{b}_{2^{2j+2}-1}(x)+ 2x^{2} + 2(x+x^{2})S_{2j}(x)\\
            & \qquad +\left(1+S_{2j+1}(x)\right)^{2} + 2xS_{2j}(x)(1+S_{2j+1}(x))\\
            &\equiv_{4}  P^{b}_{2^{2j+2}-1}(x)+ 1+2x^{2} + 2x^{2}S_{2j}(x)+2S_{2j+1}(x) \\
            &\qquad + S_{2j+1}^{2}(x)+2xS_{2j}(x)S_{2j+1}(x)\\
            & \equiv_{4} P^{b}_{2^{2j+2}-1}(x)+ 1+2x^{2} + 2x^{2}S_{2j}(x)+2S_{2j+1}(x) \\
            &\qquad + (3x+2x^{2}+2x^{3}+2x^{4}) + 2(x+x^{2})S_{2j+1}(x) + S_{2j+2}(x) + 2T_{2j+1}(x) \\
            & \qquad + 2x(S_{2j+2}(x)-x)+2xx^{2^{2j+1}}S_{2j+1}(x)
            \tag*{by \eqref{eq:s2} and \eqref{eq:ss}}\\
            & \equiv_{4} 1 + 2x^{2}+2x^{5} +2xS^{e}_{j}(x)+2S_{2j+1}(x) + S_{2j+2}(x)+2T_{2j+1}(x)\\
            & \qquad +2x^{3}S_{2j}(x)+ 2xT_{2j}(x)+2xx^{2^{2j+1}}S_{2j}(x)+2x^{2}x^{2^{2j+1}}\\
            & \equiv_{4} 1 + 2x^{2}+2x^{5} +2xS^{e}_{j}(x)+2(1+x^{3})S_{2j+1}(x) + S_{2j+2}(x) + 2(1+x)T_{2j+1}(x).
            \tag*{by \eqref{eq:t}}
        \end{align*}

        (6) By \eqref{eq:rs-p-1} and the equalities (5), (7), (9) and (13) for \(j+1\), we have 
        \begin{align*}
            P_{2^{2j+3}-2}(x) & = P_{2^{2j+2}-2}(x)P^{b}_{2^{2j+2}-2}(x) + P_{2^{2j+2}-1}(x)Q^{b}_{2^{2j+2}-2}(x)\\
            & \equiv_{4} S_{2j+1}^{2}(x)+2xS_{2j}(x)S_{2j+1}(x)+P_{2^{2j+2}-1}(x)+2x(1+x)S_{2j}(x)\\
            & \equiv_{4} (3x+2x^{2}+2x^{3}+2x^{4}) + 2(x+x^{2})S_{2j+1}(x) + S_{2j+2}(x) + 2T_{2j+1}(x)\\
            & \qquad + 2x(S_{2j+2}(x)-x)+2xx^{2^{2j+1}}S_{2j+1}(x)+P_{2^{2j+2}-1}(x)+2x(1+x)S_{2j}(x)
            \tag*{by \eqref{eq:s2} and \eqref{eq:ss}}\\
            & \equiv_{4}  2x^{2}+2x^{5}+ S_{2j+2}(x) + 2T_{2j+1}(x)+2x^{3}S_{2j+1}(x)+2xS^{e}_{j}(x)\\
            & \qquad + 2xx^{2^{2j+1}}S_{2j}(x)+2x^{2}x^{2^{2j+1}}+2x^{3}x^{2^{2j+1}} +2xT_{2j}(x)\\
            & \equiv_{4} 2x^{2}+2x^{5}+2x^{3}S_{2j+1}(x)+ S_{2j+2}(x) + 2(1+x)T_{2j+1}(x)+2xS^{e}_{j}(x).
            \tag*{by \eqref{eq:t}}
        \end{align*}

        (8) By \eqref{eq:rs-p-1} and the equalities (5), (7), (11) and (15) for \(j+1\), we have 
        \begin{align*}
            P_{2^{2j+3}-1}(x) & = P_{2^{2j+2}-2}(x)P^{b}_{2^{2j+2}-1}(x) + P_{2^{2j+2}-1}(x)Q^{b}_{2^{2j+2}-1}(x)\\
            & \equiv_{4} xP_{2^{2j+2}-2}(x) + S_{2j+1}(x)\left(P^{b}_{2^{2j+2}-1}(x)-x\right) \\
            & \qquad + xQ^{b}_{2^{2j+2}-1}(x)+\left(1+S_{2j+1}(x)\right)\left(P_{2^{2j+2}-1}(x)-x\right)\\
            & \equiv_{4} x\left(P_{2^{2j+2}-2}(x)+Q^{b}_{2^{2j+2}-1}(x)\right)+S_{2j+1}(x)\left(P^{b}_{2^{2j+2}-1}(x)+P_{2^{2j+2}-1}(x)\right)\\
            & \qquad -x-2xS_{2j+1}(x)+P_{2^{2j+2}-1}(x)\\
            & \equiv_{4} x\left(1+2xS_{2j}(x)+2x^{2^{2j+1}}\right) + S_{2j+1}(x)\left(2x+2x^{2}+2xS_{2j}(x)\right)\\
            & \qquad -x-2xS_{2j+1}(x)+P_{2^{2j+2}-1}(x)\\
            & \equiv_{4} 2(x + x^{2})x^{2^{2j+1}} + 2xS_{2j}(x)S_{2j+1}(x) \\
            & \qquad +x+2x^{2}+2x^{3}+2x^{4}+2x^{5}+2x^{3}S_{2j}(x)+2xT_{2j}(x)+2xS^{o}_{j-1}(x)\\
            & \equiv_{4} x+2x^{3}+2x^{4}+2x^{5}+2x^{3}S_{2j+1}(x)+2xT_{2j+1}(x)+2xS^{e}_{j}(x).
            \tag*{by \eqref{eq:ss} and \eqref{eq:t}}
        \end{align*} 

        (10) By \eqref{eq:rs-qb-1} and the equalities (1), (3), (9) and (13) for \(j+1\), we have 
        \begin{align*}
            Q^{b}_{2^{2j+3}-2}(x) & = Q_{2^{2j+2}-2}(x)P^{b}_{2^{2j+2}-2}(-x) + Q_{2^{2j+2}-1}(x)Q^{b}_{2^{2j+2}-2}(-x)\\
            & \equiv_{4} Q_{2^{2j+2}-2}(x)\left(P^{b}_{2^{2j+2}-2}(x)-2x\right) + Q_{2^{2j+2}-1}(x)Q^{b}_{2^{2j+2}-2}(x)\\
            & \equiv_{4} P^{b}_{2^{2j+2}-2}(x) + 2xS_{2j+1}(x) + 2(1+x)S_{2j}(x)S_{2j+1}(x) - 2x\\
            & \qquad +Q_{2^{2j+2}-1}(x) + 2(1+x)S_{2j}(x) + 2(1+x)S_{2j}
            (x)S_{2j+1}(x)\\
            & \equiv_{4} 1+2x+2(1+x)S_{2j+1}(x).
        \end{align*}

        (12) By \eqref{eq:rs-qb-1} and the equalities (1), (3), (11) and (15) for \(j+1\), we have 
        \begin{align*}
            Q^{b}_{2^{2j+3}-1}(x) & = Q_{2^{2j+2}-2}(x)P^{b}_{2^{2j+2}-1}(-x) + Q_{2^{2j+2}-1}(x)Q^{b}_{2^{2j+2}-1}(-x)\\
            & \equiv_{4} Q_{2^{2j+2}-2}(x)\left(P^{b}_{2^{2j+2}-1}(x)-2x\right) + Q_{2^{2j+2}-1}(x)\left(Q^{b}_{2^{2j+2}-1}(x)-2x\right)\\
            & \equiv_{4} P^{b}_{2^{2j+2}-1}(x)+2x^{2}+2(x+x^{2})S_{2j}(x)-2x\\
            & \qquad +\left(1+S_{2j+1}(x)\right)^{2}+2xS_{2j}(x)\left(1+S_{2j+1}(x)\right)-2x-2xS_{2j+1}(x)\\
            & \equiv_{4} 1+2x^{2}+2x^{5} + 2x^{2}x^{2^{2j+1}}+S_{2j+2}(x)+2T_{2j+1}(x)+2xx^{2^{2j+1}}S_{2j}(x)\\
            & \qquad + 2S_{2j+1}(x) + 2x^{3}S_{2j}(x) +2xT_{2j}(x) + 2xS^{o}_{j}(x)
            \tag*{by \eqref{eq:s2} and \eqref{eq:ss}}\\
            & \equiv_{4} 1+2x^{2}+2x^{5} + 2(1+x^{3})S_{2j+1}(x) + S_{2j+2}(x) + 2(1+x)T_{2j+1}(x)+2xS^{o}_{j}(x).
            \tag*{by \eqref{eq:t}}
        \end{align*}

        (14) By \eqref{eq:rs-pb-1} and the equalities (5), (7), (9) and (13) for \(j+1\), we have 
        \begin{align*}
            P^{b}_{2^{2j+3}-2}(x) & = P_{2^{2j+2}-2}(x)P^{b}_{2^{2j+2}-2}(-x) + P_{2^{2j+2}-1}(x)Q^{b}_{2^{2j+2}-2}(-x)\\
            & \equiv_{4} P_{2^{2j+2}-2}(x)\left(P^{b}_{2^{2j+2}-2}(x)-2x\right) + P_{2^{2j+2}-1}(x)Q^{b}_{2^{2j+2}-2}(x)\\
            & \equiv_{4} S_{2j+1}^{2}(x) + 2xS_{2j}(x)S_{2j+1}(x) - 2xS_{2j+1}(x) + P_{2^{2j+2}-1}(x) + 2x(1+x)S_{2j}(x)\\
            & \equiv_{4} (3x+2x^{3}+2x^{4}) + 2(x+x^{2})x^{2^{2j+1}} + S_{2j+2}(x) + 2T_{2j+1}(x)\\
            & \qquad +2xx^{2^{2j+1}}S_{2j}(x) + P_{2^{2j+2}-1}(x)  \tag*{by \eqref{eq:s2} and \eqref{eq:ss}}\\
            & \equiv_{4} 2x^{2}+2x^{5}+2x^{3}S_{2j+1}(x)+2(1+x)T_{2j+1}(x)+2xS^{o}_{j}(x) + S_{2j+2}(x). 
            \tag*{by \eqref{eq:t}}  
        \end{align*}

        (16) By \eqref{eq:rs-pb-1} and the equalities (5), (7), (11) and (15) for \(j+1\), we have 
        \begin{align*}
            P^{b}_{2^{2j+3}-1}(x) & = P_{2^{2j+2}-2}(x)P^{b}_{2^{2j+2}-1}(-x) + P_{2^{2j+2}-1}(x)Q^{b}_{2^{2j+2}-1}(-x)\\
            & \equiv_{4} P_{2^{2j+2}-2}(x)\left(P^{b}_{2^{2j+2}-1}(x)-2x\right) + P_{2^{2j+2}-1}(x)\left(Q^{b}_{2^{2j+2}-1}(x)-2x\right)\\
            & \equiv_{4} x\left(P_{2^{2j+2}-2}(x)+Q^{b}_{2^{2j+2}-1}(x)\right) + S_{2j+1}(x)\left(P_{2^{2j+2}-1}(x)+P^{b}_{2^{2j+2}-1}(x)\right)\\
            & \qquad -x-2x^{2} + P_{2^{2j+2}-1}(x)\\
            & \equiv_{4} 2x^{2}x^{2^{2j+1}} + 2xS_{2j}(x) + 2xS_{2j}(x)S_{2j+1}(x) - 2x^{2} + P_{2^{2j+2}-1}(x)\\
            & \equiv_{4} 2x^{2}x^{2^{2j+1}} + 2xx^{2^{2j+1}}S_{2j}(x) + +2x^{3}S_{2j}(x)+2xT_{2j}(x)\\
            & \qquad +x+2x^{2}+2x^{3}+2x^{4}+2x^{5}+2xS^{o}_{j}(x)
            \tag*{by \eqref{eq:ss}}\\
            & \equiv_{4} x+2x^{2}+2x^{3}+2x^{4}+2x^{5}+2xS^{o}_{j}(x) + 2x^{3}S_{2j+1}(x) + 2xT_{2j+1}(x).
            \tag*{by \eqref{eq:t}}
        \end{align*}
    Now, we have verified all the 16 equalities for \(j+1\). By induction, the result holds.
    \end{proof}

    With the help of such a subsequence of convergents, we are able to show the following congruence relation.
    \begin{proposition}\label{prop:rs-algebraic}
        The Stieltjes continued fraction \(\mathrm{Stiel}_{\mathbf{r}}(x)\) defined by the Rudin-Shapiro sequence \(\mathbf{r}\) is congruent modulo $4$ to an algebraic series in \(\mathbb{Z}[[x]]\). Namely, \(\mathrm{Stiel}_{\mathbf{r}}(x)\equiv_{4} x+2x^{2}+ 2x^{3} +(3x+2x^{3})\phi(x)+x\sqrt{1-4x\phi(x)}.\)
    \end{proposition}
    \begin{proof}
    The Stieltjes continued fraction \(\mathrm{Stiel}_{\mathbf{r}}(x)\) can be obtained by a subsequence of its convergents. Namely,
        \begin{align*}
        \mathrm{Stiel}_{\mathbf{r}}(x) & = \lim_{j\to\infty} \frac{P_{2^{2j+1}-2}(x)}{Q_{2^{2j+1}-2}(x)}\\
        & \equiv_{4} \lim_{j\to\infty} \frac{P_{2^{2j+1}-2}(x)Q_{2^{2j+1}-2}(x)}{Q^{2}_{2^{2j+1}-2}(x)} \tag*{since \(Q_{2^{2j+1}-2}(0)=1\)}\\
        & \equiv_{4} \lim_{j\to\infty} \left(P_{2^{2j+1}-2}(x)+2(1+x)S_{2j-1}(x)S_{2j}(x)\right)
        \tag*{by Lemma \ref{lem:7} (2)\ \& (6)}\\
        & \equiv_{4} \lim_{j\to\infty} \left(P_{2^{2j+1}-2}(x)+2(1+x)(S_{2j+1}(x)-x)+2(1+x)x^{2^{2j}}S_{2j}(x)\right)
        \tag*{by \eqref{eq:ss}}\\
        & \equiv_{4} \lim_{j\to\infty} \left(2x+2x^{5} + (3+2x)S_{2j}(x)+2x^{3}S_{2j-1}(x) +2xS^{e}_{j-1}(x)\right.\\
        &\qquad\qquad\quad \left.  + 2(1+x)x^{2^{2j}}S_{2j-1}(x) + 2(1+x)T_{2j-1}(x) \right)
        \tag*{by Lemma \ref{lem:7} (6)}\\
        & \equiv_{4} \lim_{j\to\infty} \left(2x+2x^{5} + (3+2x^{3})S_{2j}(x) +2xS^{o}_{j-1}(x) + 2(1+x)T_{2j}(x) \right)\tag*{by \eqref{eq:t}}\\
        & \equiv_{4} 2x+2x^{5} + (3+2x^{3})S_{\infty}(x) +2xS^{o}_{\infty}(x) + 2(1+x)T_{\infty}(x).
    \end{align*}
    Recall that \(\phi(x) = \frac{1-\sqrt{1-4x}}{2x}\). According to \eqref{eq:ct}, one has
    \[2(1+x)T_{\infty}(x)\equiv_{4} (x+x^{2})\phi(x) - (1+x)S_{\infty}(x)+2x(1+x^{2})S_{\infty}(x)+(2x^{2}+2x^{5}).\]
    Consequently, 
    \begin{align*}
    \mathrm{Stiel}_{\mathbf{r}}(x) & \equiv_{4} 2x+2x^{2} + (2+3x)S_{\infty}(x) +2xS^{e}_{\infty}(x) +(x+x^{2})\phi(x).
    \end{align*}
    It follows from \eqref{eq:ct} and \eqref{eq:s_mod_4} that
    \begin{align*}
    (2+3x)S_{\infty}(x)& \equiv_{4} 2x(1+x)\phi(x)+x\left((x\phi(x))^{4}+x+x^{2}\right)\\
    & \equiv_{4} (2x+3x^{2}+2x^{3})\phi(x) + 2x^{3}.
    \end{align*}
    By observing that \(S^{e}_{\infty}(x)^2 + S^{e}_{\infty}(x)\equiv_{2}S^{e}_{\infty}(x^2) + S^{e}_{\infty}(x)=S_{\infty}(x)\), it has been shown in \cite{HH19} that
    \[S^{e}_{\infty}(x)\equiv_{2} \frac{-1+\sqrt{1-4x\phi(x)}}{2} = \frac{-1+\sqrt{2\sqrt{1-4x}-1}}{2}.\]
    Then 
    \begin{align*}
    \mathrm{Stiel}_{\mathbf{r}}(x) & \equiv_{4} 2x+2x^{2} + (2+3x)S_{\infty}(x) +2xS^{e}_{\infty}(x) +(x+x^{2})\phi(x)\\
    & \equiv_{4} x+2x^{2}+ 2x^{3} +(3x+2x^{3})\phi(x)+x\sqrt{1-4x\phi(x)}. \qedhere
    \end{align*}
    \end{proof}
    
    The algebraicity of \(\mathrm{Stiel}_{\mathbf{r}}(x)\) yields the automaticity of \(\mathrm{Stiel}_{\mathbf{r}}\ \mathrm{mod}\ 4\).  
	\begin{proof}[Proof of Theorem \ref{thm:02}]
        The first part follows from Proposition \ref{prop:rs-algebraic}. Then using Theorem \ref{thm:ckmr}, we see that \((b_{n})_{n\ge 0}\) is \(2\)-automatic.
    \end{proof}

    \section*{Acknowledgement}
    The research was partially supported by Guangdong Natural Science Foundation (Nos.  2018A030313971, 2018B0303110005). This work was finished while the author was visiting the Department of Mathematics and Statistics, University of Helsinki. The visit was supported by China Scholarship Council (File No. 201906155024). The author would like to thank the anonymous referee and Yi-Ning Hu for valuable suggestions.

\end{document}